\documentclass[a4paper,11pt,reqno]{amsart}

\usepackage{latexsym,amssymb,amsmath,amsthm,color}
\usepackage{amscd}
\input xy
\xyoption{all} \CompileMatrices
\chardef\bslash=`\\ % p. 424, TeXbook

\numberwithin{equation}{section}

\newtheorem{theorem}{Theorem}[section]
\newtheorem{corollary}[theorem]{Corollary}
\newtheorem{lemma}[theorem]{Lemma}
\newtheorem{proposition}[theorem]{Proposition}

\newtheorem{slemma}{Lemma}[subsection]        %numeriert nach subsection
\newtheorem{sproposition}[slemma]{Proposition} %numeriert nach subsection
\newtheorem{stheorem}[slemma]{Theorem}
\newtheorem{scorollary}[slemma]{Corollary}     %numeriert nach subsection

\theoremstyle{remark}
\newtheorem{remark}[theorem]{Remark}
\newtheorem{sremark}[slemma]{Remark}          %numeriert nach subsection
\newtheorem{example}[theorem]{Example}

\newtheorem{question}[theorem]{Question}
\newtheorem{notation}[theorem]{Notation}
\theoremstyle{definition}
\newtheorem{definition}[theorem]{Definition}

\newcommand\bp{\begin{proof}}
\newcommand\ep{\end{proof}}

\newcommand{\id}{\operatorname{id}}

\newcommand{\Aut}{\operatorname{Aut}}
\newcommand{\res}{\operatorname{res}}
\newcommand{\Ad}{\operatorname{Ad}}

\newcommand{\GL}{\operatorname{GL}}

\newcommand{\DG}{{D_{P\subseteq G}}}
\newcommand{\JG}{{{\mathcal J}_{P\subseteq G}}}
\newcommand{\IG}{{{\mathcal I}_{P\subseteq G}}}
\newcommand{\DP}{{D_P}}
\newcommand{\JP}{{{\mathcal J}_P}}
\newcommand{\Om}{\Omega}

\renewcommand{\top}{\operatorname{top}}

\newcommand{\NN}{\mathbb N}
\newcommand{\ZZ}{\mathbb Z}

\newcommand{\C}{\mathbb C}
\newcommand{\CC}{\mathbb C}

\newcommand{\K}{\mathcal K}

\newcommand{\B}{\mathcal B}

\renewcommand{\span}{\operatorname{span}}

\def\lcm{\operatorname{lcm}}

\newcommand{\Hom}{\operatorname{Hom}}

\newcommand{\cA}{\mathfrak A}
\newcommand{\cB}{\mathcal B}

\newcommand{\E}{\mathcal E}
\newcommand{\cE}{\mathcal E}
\newcommand{\F}{\mathcal F}
\newcommand{\G}{\mathcal G}
\newcommand{\cH}{\mathcal H}

\newcommand{\cJ}{\mathcal J}
\newcommand{\cK}{\mathcal K}
\newcommand{\cL}{\mathcal L}

\newcommand{\cO}{\mathcal O}

\newcommand{\cR}{\mathcal R}

\newcommand{\cU}{\mathcal U}
\newcommand{\cV}{\mathcal V}

\newcommand{\cX}{\mathcal X}

\newcommand{\Spec}{\operatorname{Spec}}
\newcommand{\prolim}{\operatorname{prolim}}
\def\U{\mathcal U}

\newcommand{\blue}{\textcolor{blue}}

\newcommand{\bgl}{\begin{equation}}         %eine Gleichung mit Ziffer
\newcommand{\egl}{\end{equation}}
\newcommand{\bgln}{\begin{eqnarray}}        %mehrere Gleichungen mit Ziffer
\newcommand{\egln}{\end{eqnarray}}
\newcommand{\bglnoz}{\begin{eqnarray*}}     %mehrere Gleichungen ohne Ziffer
\newcommand{\eglnoz}{\end{eqnarray*}}
\newcommand{\btheo}{\begin{theorem}}
\newcommand{\etheo}{\end{theorem}}
\newcommand{\blemma}{\begin{lemma}}
\newcommand{\elemma}{\end{lemma}}
\newcommand{\bbew}{\begin{beweis}}
\newcommand{\ebew}{\end{beweis}}
\newcommand{\bremark}{\begin{remark}\em}
\newcommand{\eremark}{\end{remark}}
\newcommand{\bdefin}{\begin{definition}}
\newcommand{\edefin}{\end{definition}}
\newcommand{\bprop}{\begin{proposition}}
\newcommand{\eprop}{\end{proposition}}
\newcommand{\bcor}{\begin{corollary}}
\newcommand{\ecor}{\end{corollary}}

\newcommand{\mn}{\par\medskip\noindent}
\newcommand{\Ker}{{\rm Ker\,}}

\newcommand{\vp}{\varphi}
\def\Nz{\mathbb{N}}
\def\Zz{\mathbb{Z}}
\def\Fz{\mathbb{F}}
\def\Af{\mathfrak{A}[\varphi]}
\newcommand{\ti}{\tilde}

\begin{document}

\title[$K$-theory of crossed products by semigroup actions]{On the $K$-theory of crossed products by automorphic semigroup actions}
\author{Joachim Cuntz}
\author{Siegfried Echterhoff}
\author{Xin Li}
\address{Mathematisches Institut, Einsteinstr. 62, 48149
M\"unster, Germany}
\email{cuntz@uni-muenster.de}
\email {echters@uni-muenster.de}
\email{xinli.math@uni-muenster.de}

\begin{abstract}
Let $P$ be a semigroup that admits an embedding into a group
$G$. Assume that the embedding satisfies the Toeplitz condition of
\cite{Li-nuc} and that the Baum-Connes conjecture holds for $G$. We
prove a formula describing the $K$-theory of the reduced crossed
product $A\rtimes_{\alpha,r}P$ by any automorphic action of $P$.
This formula is obtained as a consequence of a result on the
$K$-theory of crossed products for special actions of $G$ on totally
disconnected spaces. We apply our result to various examples
including left Ore semigroups and quasi-lattice ordered semigroups.
We also use the results to show that for certain semigroups $P$,
including the $ax+b$-semigroup $R\rtimes R^\times$ for a Dedekind domain
$R$, the $K$-theory of the left and right regular semigroup
C*-algebras $C_\lambda^*(P)$ and $C_\rho^*(P)$ coincide, although
the structure of these algebras can be very different.
\end{abstract}

\thanks{Research supported by the Deutsche Forschungsgemeinschaft (SFB 878) and by the ERC through AdG 267079.}

\maketitle

\section{Introduction}

A semigroup (or monoid) is a set with an associative multiplication.
Recently the authors of this article - in various combinations -
have become interested in the study of the C*-algebra
$C^*_\lambda(P)$ defined by the left regular representation of a
left cancellative semigroup $P$ on the Hilbert space $\ell^2(P)$.
This interest was motivated by the fact that specific semigroups
arising from number theory give examples with an intricate, yet
tractable, structure. While generalities about semigroup C*-algebras
had been studied before by various authors, only little was known
about more complicated examples and concerning questions such as
nuclearity, $K$-theory, ideal structure etc.

The C*-algebra $C^*_\lambda(P)$ contains a natural commutative
subalgebra $D$ generated by the range projections of products of
the isometries representing the elements of $P$ and their adjoints.
These range projections correspond to the ``constructible'' right
ideals in $P$, i.e. to those right ideals that can be constructed
from the principal ideals of the form $xP$ by finitely many
operations such as intersection etc.. The spectrum of $D$ is a
totally disconnected space which we denote by $\Om_P$. Each
constructible right ideal in $P$ corresponds to a compact open
subset in $\Om_P$.

In \cite{CEL} we studied the $K$-theory of $C^*_\lambda(P)$ assuming
that $P$ satisfies the left Ore condition. This condition provides a
systematic way to embed $P$ into an enveloping group $G$ and also
allows to dilate actions of $P$ to actions of $G$, \cite{Laca}. In
particular the natural action of $P$ on $\Om_P$ can be dilated to an
action of $G$ on a totally disconnected locally compact space
$\Om_{P\subseteq G}$. The C*-algebra $C^*_\lambda(P)$ is then Morita
equivalent to the {reduced} crossed product $C_0(\Om_{P\subseteq G})\rtimes_{{r}}G$.

In \cite{CEL} we had then computed the $K$-theory (in fact in a
bivariant setting) of this crossed product using a particular
feature (``independence'', see below) of $\Om_P$ together with the
following ``descent to compact subgroups'' principle taken from
\cite{ENO},\cite{CEO}.

\begin{itemize}
\item[(DC)] Assume that $G$ satisfies the Baum-Connes conjecture with
coefficients in the $G$-algebras $A$ and $B$. Let $x$ be a class in
$KK^G(A,B)$ which induces, via descent, isomorphisms $K_*(A\rtimes
H)\cong K_*(B\rtimes H)$ for all compact subgroups $H$ of $G$. Then
$x$ also induces an isomorphism $K_*(A\rtimes_rG)\cong
K_*(B\rtimes_rG)$.
\end{itemize}

\noindent Note that, by \cite{HK}, the Baum-Connes condition
required for $G$ in (DC) holds whenever $G$ is a-$T$-menable, and
hence in particular, if $G$ is amenable.

Using the independence of the set of constructible right ideals in
$P$ and principle (DC) we determined in \cite{CEL} the $K$-theory of
$C^*_\lambda (P)$ for some prominent semigroups from algebraic
number theory. This includes the multiplicative semigroup or the
$ax+b$-semigroup for the ring of algebraic integers in a number
field or the semigroup of its principal ideals. The answer involved
well known concepts from number theory such as the ideal class group
and the group of units.

In the present paper we take a new look at the results of \cite{CEL}
from a more general perspective. We start with a general study of
group actions on totally disconnected spaces $\Om$ under an
independence condition similar to the one mentioned above. Roughly
speaking, given a totally disconnected $G$-space $\Om$ we require
that one can find a $G$-invariant family $\cV$ of compact open
subsets of $\Om$ which generates the set of all compact open sets
via finite intersections, unions and difference sets, and which is
independent in the sense that no element $U$ of $\cV$ can be written
as a finite union of elements of $\cV$ different from $U$. Let
$I=\cV\setminus\{\emptyset\}$. We are then able to construct a
canonical element $x\in KK^G(C_0(I), C_0(\Om))$ which satisfies the
requirements of (DC).

We are also able to improve the arguments used in \cite{CEL} to
allow for general coefficients. We show that for any action
$\alpha:G\to \Aut(A)$  the class ${[\id_A]}\otimes_\CC x\in
KK^G(A\otimes C_0(I), A\otimes C_0(\Om))$ will also satisfy the
conditions in (DC). Assume, then, that $G$ satisfies the Baum-Connes
conjecture with coefficients in $A\otimes C_0(I)$ and $A\otimes
C_0(\Om)$ and denote by $\tau$ resp. $\mu$ the action of $G$ on
$\Om$ resp. $I$. Using the principle (DC), we obtain an isomorphism

\begin{equation}\label{1)}K_*\big((A\otimes C_0(\Om))\rtimes_{\alpha\otimes \tau,
r}G\big)\cong K_*\big((A\otimes
C_0(I))\rtimes_{\alpha\otimes\mu,r}G\big) \end{equation}

\noindent Moreover, by Green's imprimitivity theorem the right hand
side is in turn isomorphic to the sum, over the $G$-orbits in
$I$, of the $K$-theory of the crossed products by the stabilizer
groups, i.e. to
\begin{equation}\label{2)}\bigoplus_{[i]\in G\backslash I}
K_*(A\rtimes_{\alpha,r}G_i)\end{equation} where $G_i$ denotes the
stabilizer of $i\in I$.

These results have an independent interest. Most important for us
however is again the application to the $K$-theory of semigroup
C*-algebras and semigroup crossed products. We study semigroup
crossed products $A\rtimes_{\alpha,r}P$ in which the semigroup $P$
acts {\em by automorphisms} on the C*-algebra $A$ in section
\ref{sec-semigroup}.

In \cite{Li-nuc} it was shown by the third author that, given
independence of the set of constructible right ideals, for our
purposes, the left Ore condition for $P$ can be weakened. It
suffices to assume that the semigroup $P$ is embedded into a group
$G$ and that the inclusion $P\subseteq G$ satisfies the Toeplitz
condition introduced in \cite{Li-nuc}. Under this weaker condition
too, the full and reduced C*-algebras of $P$ embed as full corners
into full and reduced crossed products by the group $G$. As we will
see, there are natural examples of semigroups satisfying the
Toeplitz condition but not the {left} Ore condition. Because of the
embedding as a full corner, again the computation of the K-theory of
a crossed product by $P$ can be reduced to the computation of the
$K$-theory of a crossed product by $G$. This crossed product by $G$
is of the form $(A\otimes C_0(\Om))\rtimes_{\alpha\otimes \tau, r}G$
considered above, and we can therefore apply formulas (\ref{1)}) and
(\ref{2)}).

We are now in a position to apply our results to explicit classes of
semigroups. Consider first a semigroup $P$ which is given as the
positive cone in a quasi-lattice ordered group $G$ which satisfies the
Baum-Connes conjecture with coefficients. The inclusion
$P \subseteq G$ satisfies the Toeplitz condition. For the crossed
product of a C*-algebra $A$ by an action $\alpha$ of $P$ by
automorphisms, we obtain the striking result
$$K_*(A)\cong K_*(A\rtimes_{\alpha,r}P)$$ i.e. the $K$-theory of the
crossed product does not depend on $P$ nor on $\alpha$. This is a
far reaching generalization of the well known corresponding result
for the action of the Toeplitz algebra by an automorphism on $A$
which in fact was the basis for the proof by Pimsner-Voiculescu of
the six term exact sequence for a crossed product by $\mathbb Z$,
\cite{PV}.

Another important example is the following. Let $R$ be the ring of
algebraic integers in a number field (or a more general Dedekind
domain). Denote by $R^\times$ its multiplicative semigroup and by
$S=R\rtimes R^\times$ its $ax+b$-semigroup. The $K$-theory of
$C^*_\lambda(S)$ was determined in \cite{CEL}. Consider now the
opposite semigroup $S^{op}$. Its left regular C*-algebra
$C^*_\lambda(S^{op})$ is the right regular C*-algebra $C^*_\rho (S)$
of $S$. We mention that $C^*_\lambda (S)$ and $C^*_\rho (S)$ are
very different algebras. For instance, the second algebra admits
non-trivial one-dimensional representations
while the first one
admits only infinite-dimensional representations. Also $S$ satisfies
the left Ore condition while $S^{op}$ does not. However, $S^{op}$
satisfies independence and the Toeplitz condition. We can therefore
again compute the $K$-theory. Somehow surprisingly, it turns out
that $C^*_\lambda (S)$ and $C^*_\rho (S)$ have the same $K$-theory, indeed they are
$KK$-equivalent.
We also determine the $K$-theory of $C^*_\lambda (S)$ and $C^*_\rho
(S)$ for a semidirect product of the form $S=H\rtimes \Nz$ where $H$
is a group. Again these two C*-algebras are completely different but
still have the same $K$-theory.

The paper is organized as follows: After a brief discussion of
totally disconnected spaces in  \S \ref{sec-prel} we present in  \S
\ref{sec-K-theory} our main results on the $K$-theory of crossed
products $(A\otimes C_0(\Om))\rtimes_rG$.  In  \S
\ref{sec-semigroup} we deduce our results on the $K$-theory of
crossed products $A\rtimes_r P$ by automorphic actions
 of semigroups and we briefly discuss the
 consequences for crossed products by the left Ore semigroups
 studied in \cite{CEL}.
 Crossed products by quasi-lattice semigroups $P\subseteq G$ are
 studied in \S \ref{sec-quasi-lattices}. Indeed, the beautiful  $K$-theory
 formula
 for such crossed products follows from the fact that for quasi-lattice
 ordered semigroups $P\subseteq G$
 the action of $G$ on the set of nonempty constructible left $P$-ideals in
 $G$ is
transitive. We present further examples which show that transitivity
of this action is not restricted to this case, and therefore
similar $K$-theory formulas can be obtained in more
generality. Our results on the left and right regular semigroup
C*-algebras $C_\lambda^*(P)$ and $C_\rho^*(P)$ are presented in \S
\ref{sec-left-right}. Finally, in the appendix we discuss some basic
constructions in equivariant $KK$-theory of finite dimensional
algebras acted upon by compact groups which we need  for checking
the principle (DC) in \S \ref{sec-K-theory}. These $KK$-results
might be known to experts, but seem not to be present in the
literature.
\medskip

{\bf Acknowledgements:} We are grateful to Marcelo
Laca for drawing our attention to \cite{Ivanov} and to Mikael
R{\o}rdam for pointing out Example \ref{ex-rordam}.

\section{Preliminaries on totally disconnected spaces}\label{sec-prel}
Recall that a locally compact Hausdorff space $\Omega$ is {\em
totally disconnected} if and only if its topology has a basis of
compact open subsets. The corresponding algebras $C_0(\Omega)$ of
continuous functions which vanish at infinity are precisely the
commutative AF-Algebras.
 In what follows, if $V\subseteq\Omega$, then $1_V:\Omega\to \CC$ denotes
 the characteristic function of $V$.

\begin{definition}\label{def-generate}
Let $\Omega$ be a totally disconnected locally compact Hausdorff space and let  $\mathcal V$ be a family of compact open subsets in $\Omega$.  Moreover, let $\cU_c(\Omega)$ denote the set of all compact open subsets of $\Omega$.
Then we say that {\em $\cV$ is a generating family of the compact open sets of $\Om$}
if $\U_c(\Omega)$ coincides with  the smallest family $\cU$ of compact open sets in $\Omega$ which contains $\mathcal V$ and which is closed under finite intersections, finite unions, and under taking differences $U\smallsetminus W$ with $U,W\in \cU$.
\end{definition}
%
%\begin{remark}\label{rem-intersections}
%If $\cV$ is a family of compact open sets in $\Omega$, then the
%smallest family $\cU$ which satisfies the conditions in the above
%definition can be constructed as the set of all subsets of $\Om$
%which can be obtained from $\mathcal V$ in finitely many steps by
%taking finite intersections, finite unions, and differences, where
%in each step we use sets constructed by either of these operations
%in the previous step. To be more precise, put $\cV_0:=\cV$ and if
%$\cV_n$ is constructed  let $\cV_{n+1}$ denote the set of all
%subsets of $\Om$ which can be obtained from elements of $\cV_n$ by
%taking finite intersections, finite unions, and differences
%$U\setminus W$ with $U,W\in \cV_n$. Then $\cU={\bigcup}_{n=0}^\infty
%\cV_n$.
%\end{remark}

\begin{lemma}\label{lem-generators}
Suppose that $\mathcal V$ is a family of compact open sets in the totally disconnected space $\Omega$. Then the following are equivalent
\begin{enumerate}
\item The set $\{1_V: V\in \mathcal V\}$ generates $C_0(\Om)$ as a C*-algebra.
\item The set $\cV$ generates $\cU_c(\Omega)$ in the sense of Definition \ref{def-generate}.
\end{enumerate}
Moreover, if $\cV$ is closed under taking finite intersections, then
(1) and (2) are equivalent to
{\begin{enumerate}
  \item[(3)]
 $\span\{1_V: V\in \cV\}$ is a dense
subalgebra of $C_0(\Om)$ containing $\span\{1_U: U\in
\cU_c(\Om)\}$.\end{enumerate}}
\end{lemma}
\begin{proof} Let $\mathcal U$ be the smallest family of compact open sets in $\Om$ 
which contains $\mathcal V$ and is closed under finite intersections, finite unions, and 
taking differences. Since a finite product of characteristic functions is
the characteristic function of the finite intersection of the given
sets, we may assume without loss of generality that $\cV$ is closed
under finite intersections. In that case it is easy to see that the
algebra generated by $\{1_V: V\in \cV\}$  coincides with
${\span}\{1_V: V\in \cV\}$. Since $1_{V_1\cup
V_2}=1_{V_1}+1_{V_2}-1_{V_1\cap V_2}$ and $1_{V_1\setminus
V_2}=1_{V_1}-1_{V_1\cap V_2}$ we see that this span contains all
characteristic functions $1_U$ with $U\in \cU$. Thus we may replace $\cV$ by $\cU$.
Note that every function in  $\span\{1_U: U\in \cU\}$ can be written
as a linear combination $\sum_{i=1}^k\lambda_i 1_{U_i}$ in which all
$\lambda_i$ are non-zero and in which the $U_i$ are pairwise
disjoint.

\noindent {Suppose} now that (1) holds. Then for every compact
open set $W$ in $\Omega$ we find  a {linear} combination
$\sum_{i=1}^k\lambda_i 1_{U_i}$ with pairwise disjoint $U_1,\ldots,
U_k$ in $\cU$ and $\lambda_i\neq 0$ such that $\|1_W-\sum_{i=1}^k
\lambda_i 1_{U_i}\|_{\infty}<\frac{1}{2}$. This implies that
{each} set $U_i$ is either a subset of $W$ or $U_i\cap
W=\emptyset$. In any case, it follows that $W$ is the union of those
$U_i$'s which are contained in $W$. Conversely, if
$\cU=\cU_c(\Omega)$, then every continuous function with compact
support can be approximated by locally constant functions with
compact supports, which are finite linear combinations of elements in
$\{1_U: U\in \cU\}$.
\end{proof}

\begin{lemma}\label{lem-projections}
Suppose that $D$ is a commutative C*-algebra such that $D$ is generated as a C*-algebra by a set of projections $\{e_i: i\in I\}\subseteq D$. Then the Gelfand spectrum $\Omega=\Spec(D)$ of $D$ is totally disconnected and the family of sets
$\cV=\{\widehat{e_i}^{-1}(\{1\}):i\in I\}$ is a family of compact open sets in $\Omega$ which generates $\cU_c(\Omega)$.
Here, for an element $d\in D$, $\widehat{d}\in C_0(\Omega)$ denotes the Gelfand-transform of $d$.
\end{lemma}
\begin{proof}
For each finite $F\subseteq I$ let $D_F\subseteq D$ denote the
C*-algebra generated by $\{e_i:i\in F\}$. Then $D_F$ is finite
dimensional  and $D=\lim_{F} D_F$. Thus, $D$ is a commutative
AF-algebra {and therefore} $\Omega=\Spec(D)$ is totally
disconnected. The second assertion then follows from Lemma
\ref{lem-generators} and the fact that projections  $e\in D$
correspond to characteristic functions $1_V\in C_0(\Omega)$ under
the Gelfand transform for $V=\widehat{e}^{-1}(\{1\})$.
\end{proof}

The above lemmas show that it is equivalent to study sets of projections $\{e_i: i\in I\}$ generating a commutative C*-algebra $D$
 or sets of compact open subsets of totally disconnected spaces $\Omega$ which generate the compact open sets $\cU_c(\Omega)$
in the sense of Definition \ref{def-generate}. For our $K$-theoretic studies we need generating sets which satisfy a certain independence condition. The following definition is taken from \cite{Li-am} and plays an important r\^ole in \cite{CEL} and
\cite{Li-nuc}.

\begin{definition}\label{def-independentset} Let $ \cJ$ be a subset of the
power set  $\mathcal P(Y)$ of a set $Y$. We call $\cJ$ {\em
independent}, if for every finite {family} $X, X_1,\ldots, X_k$
of elements in $\cJ$ such that $X={\bigcup}_{i=1}^kX_i$, there
must be an index $i\in \{1,\ldots, k\}$ such that $X_i=X$.
\end{definition}

Making the connection between sets and projections, it makes sense to extend the notion of independence to
projections in arbitrary commutative C*-algebras. We need

\begin{lemma}\label{lem-max}
Suppose that $\{e_i:i\in I\}$ is a set of projections in the commutative C*-algebra $D$. Then for each finite subset $F\subseteq I$ there exists a smallest projection $e\in D$ such that $e_i\leq e$ for all $i\in F$. We then write
$e=:\bigvee_{i\in F}e_i$.
\end{lemma}
\begin{proof} One checks that $\bigvee_{i\in F}e_i=\sum_{\emptyset\neq H\subseteq F}(-1)^{|H|-1}\prod_{i\in H}e_i$.
\end{proof}

\begin{definition}\label{def-independentproj}
Suppose that $\{e_i:i\in I\}$ is a set of projections in the commutative C*-algebra $D$.
We say that $\{e_i:i\in I\}$ is {\em independent} if for all finite  sets $F\subseteq I$ and $i_0\in I$ such that
$\bigvee_{i\in F}e_i=e_{i_0}$ it follows that $i_0\in F$.
\end{definition}

\begin{remark}\label{rem-independent}
Let $D$ be a commutative C*-algebra generated by the set of projections $\{e_i:i\in I\}$. Let $\Omega=\Spec(D)$ denote the Gelfand dual of $D$ and let $V_i:=\widehat{e_i}^{-1}(\{1\})$ for all $i\in I$. Then it is straightforward to check
that $\{e_i:i\in I\}$ is independent in the sense of Definition \ref{def-independentproj} if and only if
$\cV=\{V_i: i\in I\}$ is independent in the sense of Definition \ref{def-independentset}.
Conversely, if we start with a family $\cV$ of compact open sets in a totally disconnected space $\Omega$, then
$\cV$ is independent if and only if the set $\{1_V: V\in \cV\}$ is an independent set of projections.
\end{remark}

{The following lemma is obvious, but also follows from \cite[Proposition 2.4]{Li-nuc}:}

{\begin{lemma}\label{lem-lin-independent}
Suppose that $\{e_i:i\in I\}$ is a 
 family of projections in the commutative C*-algebra $D$ which is closed under 
multiplication up to $0$. Then $\{e_i:i\in I\}$ is independent in the sense of Definition \ref{def-independentproj}
if and only it is linearly independent.
\end{lemma}}

\begin{definition}\label{def-regular}
Suppose that $\Om$ is a totally disconnected locally compact
Hausdorff space. A family  $\cV$ {of} non-empty compact open
subsets of $\Om$ is called a {\em regular basis} (for the compact
open sets in $\Om$) if the following are satisfied:
\begin{enumerate}
\item $\cV\cup\{\emptyset\}$ is closed under finite intersections;
\item $\cV$   generates the compact open sets of $\Om$;
\item $\cV$ is independent.
\end{enumerate}
Similarly, if $\{e_i:i\in I\}$ is a set of non-zero projections in a commutative C*-algebra $D$, we say that $\{e_i: i\in I\}$ is a {\em regular basis}
for $D$ if it is {(linearly)} independent, closed under multiplication (up to $0$) and generates $D$ as a C*-algebra, which by
Lemma \ref{lem-generators} implies that  $\span\{e_i:i\in I\}$ is a dense subalgebra of $D$.
\end{definition}

We have the following countability result for totally disconnected
spaces. Recall that a topological space $\Om$ is called {\em second
countable} if  it has  a countable {basis} for its topology.

\begin{lemma}\label{lem-countable}
Let $\Om$ be a totally disconnected locally compact Hausdorff space. Then $\Om$ is second countable if and only if
the set $\cU_c(\Om)$ of compact  open subsets of $\Om$ is countable.
\end{lemma}
\begin{proof} If $\Om$ is  second countable we can find a countable {basis}  $\cU$  for the topology of $\Om$ consisting of
compact open subsets of $\Om$. But then each compact open subset of $\Om$ is a finite union of elements in $\cU$, which shows that $\cU_c(\Om)$ is countable. The converse is clear.
\end{proof}

\begin{remark}\label{rem-countable}
It follows from the above lemma that if  $\Om$ is a second countable totally disconnected locally compact Hausdorff space, then every regular {basis} $\cV$ for the compact open sets of $\Om$ is countable.
\end{remark}

 For second countable spaces $\Om$ we can prove
the existence of a regular {basis} for the compact open sets in $\Om$:

\begin{proposition}\label{prop-exists}
Let $\Om$ be a second countable totally disconnected locally compact space. Then there exists a regular {basis} $\cV$ for the compact open sets of $\Om$.
\end{proposition}
\begin{proof}
We first observe that it suffices to consider the case where $\Om$
is compact. This follows from the fact that every locally compact
totally disconnected space $\Om$ can be written as the disjoint
union of compact open sets $\{\Om_i:i\in I\}$. Then, if $\cV_i$ is a
regular basis for the compact open sets of $\Om_i$ for all $i\in I$,
then $\cV={\bigcup}_{i\in I} \cV_i$ is a regular {basis} for the compact
open sets in $\Om$.

So assume from now on that $\Om$ is compact. Since  $\Om$ is second
countable, it can be  realized as a projective limit
$\Om=\prolim_{n\in \NN} F_n$ for some projective system $\{F_n;\;
\varphi_n:F_{n+1}\to F_n\}$ in which all sets $F_n$ are finite.
Recall from the construction of this projective limit that a basis
$\cU$ of the topology of $\Om$ consisting of compact open sets is
given by $\cU=\{\mu_n^{-1}(x): n\in \NN, x\in F_n\}$, where, for
each $n\in \NN$,
 $\mu_n:\Om\to F_n$ denotes the canonical mapping.

In order to construct a {regular} {basis} $\cV$ for the compact open sets of
$\Om$ we first construct bijections $\psi_n:\{1,\ldots, k_n\}\to
F_n$, with $k_n=|F_n|$, which satisfy the following compatibility
condition:
 \begin{itemize}
 \item[{(C)}] For each $n\in \NN$ let $m_0:=0$ and
 $m_l:=|\varphi_n^{-1}(\psi_n(\{1,\ldots, l\}))|$ or $l\in \{1,\ldots, k_n\}$.
 We require that $\varphi_n:F_{n+1}\to F_n$ sends $\psi_{n+1}\big(\{m_{l-1}+1, \ldots, m_l\}\big)$ to $\psi_n(l)$ for all $l\in \{1,\ldots, k_n\}$.
 \end{itemize}
 The construction can be done easily by starting with an arbitrary
 bijection $\psi_1:\{1,\ldots, k_1\}\to F_1$ and then defin{ing} the other
 bijections recursively by obeying condition (C) in each step. Having done
 this, we may assume as well that $F_n=\{1, \ldots, k_n\}$  and that
 $\varphi_n(\{m_{l-1}+1, \ldots, m_l\})=\{l\}$ for each
 $1\leq l\leq k_n$.

We then define $\cV:=\{V_{n,l}:=\mu_n^{-1}(\{1,\ldots, l\}): n\in
\NN,1\leq l\leq k_n\}$. To see that this is a regular {basis} for the
compact open sets of $\Om$ we first observe that each basic open set
$\varphi_n^{-1}(\{l\})$ can be obtained as a difference of two sets
in $\cV$, so it is clear that $\cV$ generates the compact open sets
of $\Om$. To check the other conditions,  observe first  that
condition (C) together with the equation $\varphi_n\circ
\mu_{n+1}=\mu_n$  implies that $V_{n,l}=V_{n+1, m_l}$ for all $n\in
\NN, 1\leq l\leq k_n$, with $m_l$ as in (C). By induction, it
follows that $V_{n,l}=V_{m,l'}$ for some suitable $1\leq l'\leq k_m$
whenever, $m\geq n$. So, if finitely many elements $W_1, \ldots,
W_r$ in $\cV$ are given, we may assume that there exist $n\in \NN$
and $1\leq l_1\leq  l_2 \leq\cdots\leq l_r\leq k_n$ such that
$W_i=V_{n, l_i}$ for all $1\leq i\leq r$. The intersection of these
sets then equals $V_{n, l_1}$. {This} proves that $\cV$ is closed
under finite intersections. {The union of the $W_i$} equals
$V_{n, l_r}$, which proves independence.
 \end{proof}

We close this section with a {simple} example which illustrates the concept of regular {bases} for the compact open sets of a totally
disconnected space $\Om$.

 \begin{example}\label{ex-cantor}
 Consider the space $\Om=\{1,-1\}^{\ZZ}$ equipped with the product topology. Then
 $\Om$ is homeomorphic to the Cantor space. Recall that the basic open neighborhoods of an element $x=(x_n)_{n\in \ZZ}\in \Om$
 are given by the sets $W_F(x):=\{y\in \Om: y_n=x_n\;\text{for all}\; n\in F\}$, where $F$ runs through the finite subsets of $\ZZ$.

 For every finite set $F\subseteq \ZZ$ (including $\emptyset$) we define
 $V_F:=\{z\in \Om: z_n=1\;\text{for all}\; n\in F\}$ and we let $\cV$ denote the family of all such sets $V_F$.
Since $V_{F_1}\cap V_{F_2}=V_{F_1\cup F_2}$ we see  that $\cV$ is closed under finite intersections.
To see that it is independent, observe that for finite sets $F_1,\ldots, F_l$ we have
$$V_{F_1}\cup V_{F_2}\cup\cdots\cup V_{F_l}=\{z\in \Om: \exists k\in \{1,\ldots,l\}\;\text{such that}\; z_n=1\;\text{for all}\; n\in F_i\}$$
which is equal to a set $V_F$ if and only if there exists $i_0\in \{1,\ldots, k\}$ such that $F=F_{i_0}$ and $F_i\subseteq F$ for all
$i\in \{1,\ldots, l\}$. Thus it follows that
$\cV$ is a regular basis for the compact open sets of $\Om$ if it generates the compact open sets $\cU_c(\Om)$ of
$\Om$. For this let $\cU$ denote the smallest subset of $\cU_c(\Om)$ which contains $\cV$ and is closed under taking differences, finite intersections and finite unions. It suffices to show that $\cU$ contains
 all basic neighborhoods $W_F(x)$.   To see this we first observe that $\Om=V_{\emptyset}\in \cV$ .
 Then for any fixed $n_0\in \ZZ$ the complement
 $V_{n_0}^-:=\Om\smallsetminus V_{n_0}=\{z\in \Om: z_{n_0}=-1\}$ lies in $\cU$. For a given finite subset
 $F$ of $\ZZ$ and any given $x\in \Om$ we then have
 $$W_F(x)=(\bigcap\{V_{n}: n\in F, x_n=1\})\cap(\bigcap\{V_{m}^-: m\in F, x_m=-1\}),$$
 so $W_F(x)\in \cU$.
 \end{example}

\section{$K$-theory of crossed products by actions on totally disconnected spaces}  \label{sec-K-theory}
{In this section we extend the ideas of \cite[\S 6]{CEL} to
study the $K$-theory of crossed products of the form
$C_0(\Om)\rtimes_{\tau,r}G$ for a continuous action of a second
countable locally compact group $G$ on a second countable totally
disconnected locally compact space $\Om$. More generally, we study
the $K$-theory of a crossed product $(A\otimes
C_0(\Om))\rtimes_{\alpha\otimes\tau,r}G$ by a diagonal action where
$\alpha:G\to \Aut(A)$ is an action of $G$ by $*$-automorphisms on a
separable C*-algebra $A$. We will assume that we can find a
$G$-invariant regular basis $\cV$ for the compact open sets in
$\Om$. Moreover, we will use the assumption that $G$ satisfies the
Baum-Connes conjecture for suitable coefficients (see the discussion
below).}

At the end of this section we will use the $K$-theoretic results of this
section to show that a $G$-invariant regular basis for the compact open sets of $\Om$ cannot always exist (see Examples \ref{ex-notexists} and \ref{ex-rordam} below). But the results in \cite{Li-nuc} show that such a basis does exist in many interesting situations connected to the study of crossed products by  semigroups {(e.g., see \S \ref{sec-quasi-lattices} for explicit examples)}.
 Let us give a first positive example:

\begin{example}\label{ex-cantor-action}
Consider  the Cantor set $\Om=\{1,-1\}^{\ZZ}$ of Example \ref{ex-cantor}.  Then $\ZZ$ acts on $\Om$ by the shift, i.e.,
$(m\cdot x)_n:=x_{n-m}$ for $m\in \ZZ$ and  $x=(x_n)_{n\in \ZZ}\in \Om$. It is then  clear that the
regular basis $\cV=\{V_F: F\subseteq \ZZ\;\text{finite}\}$ as constructed in Example \ref{ex-cantor} is $\ZZ$-invariant.
\end{example}

{F}rom now on we assume that $\cV=\{V_i:i\in I\}$ is  a
$G$-invariant regular basis for the  compact open sets in $\Om$. We
then may assume without loss of generality that $G$ acts on the
index set $I$ via a homomorphism $\mu:G\to S_I$ of $G$  into the
group of permutations of $I$ such that $g\cdot V_i=V_{g i}$ for all
$i\in I$ and $g\in G$. Note that it follows from Lemma
\ref{lem-countable} that $I$ is countable (we always assume that the
assignment $i \mapsto V_i$ is bijective). In what follows, we  equip $I$
with the {\em discrete} topology.

\begin{remark}\label{rem-stabilizer}
We should remark that, although $G$ is not assumed to be discrete, the action of $G$ on $I$ is automatically continuous, which just means that the stabilizers $G_i=\{g\in G: gi=i\}$ are open in $G$ for all $i\in I$. This follows from the fact that
$G_i$ coincides with the stabilizer  $G_{1_{V_i}}=\{g\in G: \tau_g(1_{V_i})=1_{V_i}\}$ for the function $1_{V_i}$ under
the continuous $\tau$ action of $G$ on $\Om$. But $G_{1_{V_i}}=\{g\in G: \|\tau_g(1_{V_i})-1_{V_i}\|_\infty<1\}$ which is open in $G$.
\end{remark}

We are  going to construct a class
$x\in KK^G(C_0(I), C_0(\Om))$ which, under some extra condition on $G$ which we explain below,
induces via descent an isomorphism
$$K_*\big((A\otimes C_0(I))\rtimes_{\alpha\otimes\mu,r}G\big)\cong K_*\big((A\otimes C_0(\Om))\rtimes_{\alpha\otimes \tau,r}G\big),$$
and, in  good cases, even a $KK$-equivalence between
these algebras. The relevant extra conditions are related to the
Baum-Connes conjecture {for} the group $G$, which, in case it
holds, identifies the $K$-theory  of a reduced crossed product
$B\rtimes_{\beta,r}G$ with the topological $K$-theory
$K_*^{\top}(G;B)$ of $G$ with coefficients in $B$. To be more
precise, for every C*-dynamical system $(B,G,\beta)$ there is a
canonical {\em assembly map}
$$\mu_\beta: K_*^{\top}(G;B)\to K_*(B\rtimes_{\beta,r}G)$$
and we say that $G$ satisfies the Baum-Connes conjecture for $B$ if
this map is an isomorphism. By work of Higson and Kasparov
\cite{HK}, the Baum-Connes conjecture holds for all $G$-algebras $B$
if one can find a proper $G$-algebra $\mathcal A$ which is
$G$-equivariantly $KK$-equivalent to $\CC$. (Recall that $\mathcal
A$ is called a proper $G$-algebra if there exists a locally
compact proper $G$-space $X$ such that  there exists a nondegenerate
$G$-equivariant $*$-homomorphism $\Phi:C_0(X)\to ZM(\mathcal A)$.)
In this case we say that $G$ satisfies  the {\em strong Baum-Connes
conjecture} with arbitrary coefficients. For our purposes we do not
need to know anything about the definition of the  topological
$K$-theory group, but the interested reader is referred to
\cite{BCH} for an introduction to this interesting theory.

The  result which is important for us is the  following
proposition. It is taken from \cite{ENO}, but is based on earlier
work in \cite{CEO, MN, ELPW}, and gives a more detailed formulation of the
principle {(DC)} of the introduction:

\begin{proposition}\label{prop-BC}
Let $A$ and $B$ be $G$-algebras and let $x\in KK^G(A,B)$.
Let $j_G(x)\in KK_0(A\rtimes_{\alpha,r}G, B\rtimes_{\beta,r}G)$ denote the descent of $x$
for the reduced crossed products. For every compact subgroup $H$ of $G$ let
$$\varphi_H: K_*(A\rtimes_\alpha H)\to K_*(B\rtimes_\beta H);\;\;\varphi_H(y)=y \otimes j_H(\res_H^G(x)).$$
where ``$\otimes$'' denotes the Kasparov product.
Then the following are true:
\begin{enumerate}
\item If $G$ satisfies the Baum-Connes conjecture for $A$ and $B$ and if
$\varphi_H$ is an isomorphism for every compact subgroup
$H$ of $G$, then $\cdot\otimes {j_G(x)}:K_*(A\rtimes_{\alpha,r}G)\to K_*(B\rtimes_{\beta,r}G)$ is an isomorphism.
\item If $G$ satisfies the strong Baum-Connes conjecture and if  $j_H(\res_H^G(x))$
is a $KK$-equivalence between $A\rtimes_{\alpha} H$ and $B\rtimes_\beta H$ for all compact subgroups
$H$ of $G$, then
$j_G(x)$ is a $KK$-equivalence between $A\rtimes_{\alpha,r} G$ and $B\rtimes_{\beta,r} G$.
\end{enumerate}
\end{proposition}
{The Baum-Connes conjecture with coefficients in arbitrary
$G$-algebras admits a counter-example (see \cite{HLS}). On the other
hand, the validity of the conjecture has been checked for many
interesting classes of groups. One of the strongest results is given
by Higson and Kasparov in \cite{HK} where they show that all
a-$T$-menable groups (this includes all amenable groups and all
countably generated free groups) satisfy the strong Baum-Connes
conjecture. }

For the construction of $x$ we start with homomorphisms
$\varphi_i:\CC\to C_0(\Om)$ which map $1\in \CC$ to the projection
$e_i:=1_{V_i}\in C_0(\Om)$. This gives a class in $KK(\CC,
C_0(\Om))$. Viewing now this copy of $\CC$ as the $i$th component of
$C_0(I)=\bigoplus_{i\in I}\CC$ and using the well-known isomorphism
$$KK(C_0(I), C_0(\Om))\cong \prod_{i\in I}KK(\CC, C_0(\Om))$$
we obtain a class $x\in KK(C_0(I), C_0(\Om))$. We need to make this
class $G$-equivariant.  This is a special case of the following
general construction:

\begin{notation}\label{not-prod}
Suppose that $C=\bigoplus_{i\in I}C_i$ is a direct sum of  C*-algebras $C_i$ and suppose that
 for all $i\in I$ we have a homomorphism $\varphi_i:C_i\to B$ into some fixed C*-algebra $B$.
Then there is a
$KK$-class $x\in KK(C, B)$  given by the Kasparov-triple $(\mathcal E,\varphi, 0)$
with $\mathcal E=\ell^2(I)\otimes B$ (with grading given by $\mathcal E_0=\mathcal E, \mathcal E_1=\{0\}$) equipped with the canonical $B$-valued inner product and with
$$\varphi: C\to \mathcal K(\mathcal E)\cong \K(\ell^2(I))\otimes B; \;\;\varphi=\bigoplus_{i\in I}\varphi_i.$$

Alternatively, $x$ is represented by the $*$-homomorphism $\varphi:C\to \K(\ell^2(I))\otimes B$ via the identification $KK(C, \mathcal\K\otimes B)\cong KK(C,B)$ given by multiplication with the $KK$-class $m_I\otimes \id_B$, where $m_I=[(\ell^2(I), \id_{\mathcal K}, 0)]\in KK(\K(\ell^2(I)), \CC)$ denotes the class of the canonical Morita equivalence $\K(\ell^2(I))\sim_M\CC$.

Suppose, moreover, that $\gamma:G\to \Aut(C)$, $\beta:G \to \Aut(B)$ are actions such that $\gamma$ induces an action $\mu:G\to S_I$ of $G$ on $I$ by permutations and
such that $\varphi$ becomes $G$-equivariant with respect to the action  $\Ad \mu\otimes \beta$ on $\mathcal K(\ell^2(I))\otimes B$. Then the action $\mu\otimes\beta:G\to \Aut(\ell^2(I)\otimes B)$ turns $x$ into a class in $KK^G(C,B)$.
(Note that the corresponding class in $KK^G\big(C, \K(\ell^2(I))\otimes B\big)$ is just given by the $G$-map $\varphi:C\to \K(\ell^2(I))\otimes B$.)

We use this to construct equivariant $KK$-elements as follows:
\begin{enumerate}
\item[{(E1)}] {Let}  $B=C_0(\Om)$,
$C={C_0(I)}=\bigoplus_{i\in I}\CC$ and {let}
$\varphi_i:\CC\to C_0(\Om)$ {be} given by
$\varphi_i(1)=e_i=1_{V_i}$ as above. Then it is straightforward to
check that the resulting homomorphism $\varphi=\bigoplus_{i\in
I}\varphi_i: C_0(I)\to \K(\ell^2(I))\otimes C_0(\Om)$ is
$G$-equivariant as required in the previous paragraph, and we obtain
a class $x\in KK^G(C_0(I), C_0(\Om))$.
\item[{(E2)}] {More generally, let $\alpha:G\to \Aut(A)$ be an action
of $G$ on a C*-algebra $A$. Consider the case $B=A\otimes C_0(\Om)$
and $C=A\otimes C_0(I)\cong \bigoplus_{i\in I}A$,
$\psi_i=\id_A\otimes \varphi_i: A\otimes \CC\to A\otimes C_0(\Om)$.
Then $\psi:=\bigoplus_{i\in I}\psi_i:A\otimes C_0(I)\to
\cK(\ell^2(I))\otimes A\otimes C_0(\Om)$ can be identified with
$\id_A\otimes \varphi$ after applying the flip isomorphism
$\cK(\ell^2(I))\otimes A\otimes C_0(\Om)\cong A\otimes
\cK(\ell^2(I))\otimes C_0(\Om)$. Hence the corresponding class in
$KK^G(A\otimes C_0(I), A\otimes C_0(\Om))$ coincides with
${[\id_A]}\otimes_{\CC}x$ where $x\in KK^G(C_0(I), C_0(\Om))$ is as in
{(E1)}.}
\end{enumerate}
\end{notation}

We need some observations regarding this construction. We start with

\begin{lemma}\label{lem-limit} Let $H$ be a closed subgroup of $G$ and
suppose that $C=\bigoplus_{i\in I}C_i$ and $x\in KK^G(C,B)$ are as in the general construction above. Then for each $H$-invariant subset $J\subseteq I$ let $C_J:=\bigoplus_{i\in J}C_i\subseteq C$ and let $B_J\subseteq B$ denote the smallest $H$-invariant C*-subalgebra of $B$ which contains all  images $\varphi_i(C_i)$, $i\in J$.
Then the above construction applied to $C_J, B_J$ and $H$ gives a class
$$x_J=[(\ell^2(J)\otimes B_J, \varphi_J, 0)]\in KK^H(C_J, B_J),$$
with $\varphi_J=\bigoplus_{i\in J}\varphi_i$. Moreover, if
$\iota^C_J: C_J\to C$ and $\iota_J^B: B_J\to B$ denote the inclusions, then
$$[\iota^C_J]{\otimes}_C \res_H^G(x)= x_J\otimes_{B_J}[\iota_J^B]\in KK^H(C_J,B).$$
\end{lemma}
\begin{proof}
The product $x_J\otimes_{B_J}[\iota_J^B]$ is represented by the Kasparov triple
$\big(\ell^2(J)\otimes B_J)\otimes_{B_J}B, \varphi_J\otimes 1_B, 0\big)\cong \big(\ell^2(J)\otimes B, (1_{\ell^2(J)}\otimes \iota_J^B)\circ \varphi_J,0\big)$, while the product $[\iota_J^C]\otimes_{C_J} \res_H^G(x)$ is represented by the
triple $\big(\ell^2(I)\otimes B, (\iota_J^\K\otimes \iota_J^B)\circ \varphi_J, 0\big)$, where
$\iota_J^\K:\K(\ell^2(J))\to\K(\ell^2(I))$ denotes the canonical inclusion.
Since both triples differ by the degenerate triple $\big(\ell^2(I\setminus J)\otimes B, 0, 0\big)$, the result follows.
\end{proof}

We {note} that in the alternative picture where we regard $x$ as an element in \linebreak $KK^G(C, \K(\ell^2(I))\otimes B)$ and $x_J$ as an element in $KK^H(C_J, \K(\ell^2(J))\otimes B_J)$,  the
equation of the above lemma {translates} into the equation
$$[\iota_J^C]\otimes_{C_J}\res_H^G(x)=x_J\otimes_{\K(\ell^2(J))\otimes B_J}[\iota_J^\K\otimes\iota_J^B].$$
 {This} follows  from the equation $\varphi\circ \iota_J^C=(\iota_J^\K\otimes\iota_J^B)\circ \varphi_J$  (which is easily checked on each summand $C_i$ of $C_J$). In the following lemma, $K_*^H(C)=KK_*^H(\CC, C)$ denotes  $H$-equivariant $K$-theory for the {\em compact} subgroup $H$ of $G$. Note that since $I$ is discrete, it follows from Remark \ref{rem-stabilizer} that the orbits for the action
of $H$ on $I$ are automatically finite.

\begin{lemma}\label{lem-finite}
Suppose that $H\subseteq G$ is a compact subgroup and let $\mathcal F$ denote the set of all finite $H$-invariant subsets
of $I$ ordered by inclusion. Then $K^H_*(C)=\lim_{J\in \mathcal F} K_*^H(C_J)$,
$K_*^H(B_I)=\lim_{J\in \mathcal F} K_*^H(B_J)$ and we obtain a commutative diagram
$$
\begin{CD} \lim_{J\in \mathcal F} K_*^H(C_J) @>\lim_{J\in \mathcal F} {([\cdot] \otimes_{C_J} x_J)}>> \lim_{J\in \mathcal F} K_*^H(B_J)\\
@V\cong VV   @VV\cong V\\
K_*^H(C)  @>>[\cdot]\otimes_C \res_H^G(x)> K_*^H(B_I).
\end{CD}
$$
In particular, if all maps $[\cdot]\otimes_{C_J}x_J:K_*^H(C_J) \to K_*^H(B_J)$ are isomorphisms, the same is true
for $[\cdot] {\otimes_C \res_H^G(x)} :K_*^H(C) \to K_*^H(B_I)$
\end{lemma}
\begin{proof} We note first that equivariant $K$-theory $K_*^H(C)$ is continuous for  compact groups $H$, since it can be identified with $K_*(C\rtimes H)$ via the Green-Julg theorem, and hence continuity follows from continuity of ordinary $K$-theory. Now the previous lemma implies commutativity of the diagram
$$
\begin{CD} K_*^H(C_J) @>[\cdot] \otimes_{C_J} x_J>>  K_*^H(B_J)\\
@V\iota_J^C VV   @VV\iota_J^B V\\
K_*^H(C)  @>>[\cdot]\otimes_C \res_H^G(x) > K_*^H(B_I).
\end{CD}
$$
and the result then follows from taking limits.
\end{proof}

\begin{remark}\label{rem-descent}
Suppose that $H$ is a compact group and $x\in KK^H(C,B)$. Let $j_H(x)\in KK(C\rtimes H, B\rtimes H)$ denote the descent of $x$. Recall that the Green-Julg isomorphism
$$\mu_C^H: K^H_*(C)=KK^H_*(\CC,C)\stackrel{\cong}{\to} K_*(C\rtimes H)$$
can be described as the composition
$$KK^H_*(\CC,C)\stackrel{j_H}{\to}  KK_*(C^*(H), C\rtimes H) \stackrel{p^*}{\to} KK(\CC, C\rtimes H)$$
where $p:\CC\to C^*(H)$ sends $1\in \CC$ to the projection $1_H\in C(H)\subseteq  C^*(H)$ (which is the projection
corresponding to the trivial representation $1_H$ of $H$ in the Peter-Weyl decomposition of $C^*(H)$).
Then the diagram
$$
\begin{CD} K_*^H(C) @> [\cdot]\otimes x>> K_*^H(B)\\
@V\mu_C^HVV  @VV\mu_B^HV\\
K_*(C\rtimes H) @>>[\cdot]\otimes {j_H(x)} > K_*(B\rtimes H)
\end{CD}
$$
commutes. This follows from the fact that $j_H$ preserves Kasparov products, and hence
$$
\mu_C^H(y)\otimes_{C\rtimes H}j_H(x)
=([p]\otimes_{C^*(H)} (j_H(y))\otimes_{C\rtimes H}j_H(x)$$
$$\quad\quad\quad\quad\quad\quad\quad\quad\quad\quad=[p]\otimes_{C^*(H)}j_H(y\otimes_Cx) \\
=\mu_B^H(y\otimes_Cx)
$$
for all $y\in KK_*^H(\CC,C)$. We therefore may replace
$H$-equivariant $K$-theory by the $K$-theory of the corresponding crossed products
everywhere in the above lemma. In particular, we see that  if all maps $[\cdot]\otimes_{C_J}x_J:K_*^H(C_J) \to K_*^H(B_J)$ in that lemma are isomorphisms, then the map
$$[\cdot]\otimes j_H(x):K_*(C\rtimes H)\to K_*(B_I\rtimes H)$$
is an isomorphism, too.
\end{remark}

We are now coming back to the special situation where the
commutative C*-algebra $D=C_0(\Om)$ is generated by the collection
$\{e_i=1_{V_i}:i\in I\}$ of projections corresponding to the
$G$-invariant regular basis $\cV=\{V_i:i\in I\}$. Since $\cV$ is
closed under finite intersections (up to $\emptyset$) it follows
that the family of projections $\{e_i:i\in I\}$ is invariant under
multiplication (up to $0$). Let us consider the case where $I$ is
finite:

\begin{lemma}\label{lem-nilpotent}
Let $D$ be a commutative C*-algebra generated by a multiplicatively
closed (up to 0) and independent finite family of projections
$\{e_i: i\in I\}$. For each $i\in I$ let
$e_i':=e_i-\bigvee_{e_j<e_i}e_j$. Then $\{e_i'\}$ is a family of
nonzero pairwise orthogonal projections spanning $D$. The
transition matrix $\Gamma=(\gamma_{ij})$ determined by the equations
$e_j=\sum_{i\in I}\gamma_{ij}e'_i$, is unipotent and therefore
invertible over $\Zz$.
\end{lemma}
\begin{proof} Independence shows that $e_i'\neq 0$ for all $i$. For
$i\neq j$ we have at least one of both, $e_i'e_j=0$ or $e_ie_j'=0$.
Both equalities imply $e_i'e_j'=0$. Since  $\dim(D)\,\leq |I|$, the
$e_i'$ linearly span $D$.\\
If $e_i'\leq e_j$, then $e_i'\leq e_ie_j\leq e_i$, whence
$e_ie_j=e_i$ by definition of $e_i'$. This shows that
$\gamma_{ij}=1$ if $e_i\leq e_j$ and $\gamma_{ij}=0$
otherwise. Thus, for the partial ordering $i\leq j\Leftrightarrow
e_i\leq e_j$, the matrix $\Gamma$ is upper triangular with 1's on
the diagonal. Thus $1-\Gamma$ is nilpotent of order $|I|$ (because
there are no strictly increasing sequences of length $\geq |I|$ in
$I$). It follows that $\Gamma$ is invertible with inverse
$\sum_{n=0}^{|I|} (1-\Gamma)^n$.
\end{proof}

\begin{remark}\label{rem-invertible}
Let $C$ and $B$ be two finite dimensional commutative
C*-algebras with bases $\{c_1,\ldots, c_n\}$ and $\{b_1,\ldots,
b_{{m}}\}$ consisting of pairwise orthogonal projections and
equipped with actions of $H$ given by permutations of the bases
induced by homomorphisms $\mu_C: H\to S_n, \mu_B:H\to S_m$. In the
appendix we show that every $H$-equivariant matrix $\Gamma\in
M(m\times n, \ZZ)$ gives rise to a canonical element
$$x_\Gamma\in KK^H(C,B)$$
which by {Lemma \ref{lem-compose}} is invertible if and only if
$\Gamma$ is invertible. We want to compare that construction with
the construction of the element $x_J\in KK^H(C_0(J), D_J)$ in which
$J\subseteq I$ is a finite $H$-invariant set and $D_J\subseteq
C_0(\Om)$ is the subalgebra of $C_0(\Om)$ generated by $\{e_i:i\in
J\}$ {which we assume to be closed under multiplication up to $0$.}

Let $\{e_i': i\in J\}$ be the set of orthogonal projections
constructed from $\{e_i:i\in J\}$ as in the above lemma and let
$\{c_i: i\in J\}$ be the standard {basis} of $C_0(J)$.
Identifying $J$ with $\{1,\ldots, n\}$  for $n=|J|$, we see from the
above lemma that the transition matrix $\Gamma$  for passing  from
$\{e_i:i\in J\}$ to $\{e'_i:i\in J\}$ is invertible and has only
entries $0$ or $1$. Moreover, the element $x_J\in KK^H(C_0(J), D_J)$
coming from our general construction with $C=C_0(J)$ and $B=D_J$ is
given by the Kasparov cycle $[\cE_J, \varphi_J, 0]$ with $\mathcal
E_J=\ell^2(J)\otimes {D_J}=\bigoplus_{j=1}^n D_J$ and in which
$\varphi_J(c_j)$
 acts via the projection $e_j$ in the $j$th component of this direct sum and
 as $0$ in all other components. Thus we may restrict the module  to the nondegenerate
 part $\E_0:= \varphi_J(C_0(J))\cE_J$, which is $\cE_0=\bigoplus_{j=1}^n e_j D_J$.

On the other hand, the element $x_\Gamma\in KK^H(C_0(J), D_J)$ constructed in the appendix  from the transition matrix
$\Gamma$
and the {bases} $\{c_1,\ldots, c_n\}$ of $C_0(J)$ and $\{e'_1,\ldots, e'_n\}$ of $D_J$
is given  by the Kasparov cycle $[\cE, \psi, 0]$ in which
$\cE=\bigoplus_{j=1}^n\left(\bigoplus_{i=1}^n(\CC^{\gamma_{ij}} \otimes \CC e_i')\right)$
and where $\psi(c_j)$  acts via the projection of the $j$-th summand  $\bigoplus_{i=1}^n (\CC^{\gamma_{ij}}\otimes \CC e_i')$
of this module. Now, since $e_j=\sum_{i=1}^n \gamma_{ij}e_i'$ and $\gamma_{ij}$ only takes values $0$ or $1$,  one easily checks that $e_j D_J\cong \bigoplus_{i=1}^n\gamma_{ij}e_i'D_J\cong \bigoplus_{i=1}^n(\CC^{\gamma_{ij}} \otimes \CC e_i')$ as Hilbert $D_J$-modules and that this isomorphism intertwines $\varphi(c_j)$ with $\psi(c_j)$ for all $1\leq j\leq n$.  This
proves  $x_J=x_\Gamma\in KK^H(C_0(J), D_J)$. In particular, it follows that $x_J$ is invertible!

\end{remark}

We are now ready for the main result of this section.

%So assume that $D$ is a commutative
%C*-algebra generated by the independent and multiplicatively closed (up to $0$) family of projections
%and that $\{e_i:i\in I\}$.
%Suppose further that the discrete group $G$ acts on $I$ via a homomorphism $\mu:G\to S_I$ and that
%$\tau:G\to \Aut(D)$ is an action of $G$ on $D$ such that $\tau_g(e_i)=e_{gi}$ for all $i\in I$.
% Let $x\in KK^G(C_0(I), D)$ be the element constructed from these data as in Notation \ref{not-prod}. We then get

\begin{proposition}\label{prop-finite}
Let $x\in KK^G(C_0(I), C_0(\Om))$ be as in Notation \ref{not-prod}
(N1) and let $A$ be any $G$-algebra.
Then for any compact  subgroup $H\subseteq G$  the
restriction
$$\res_H^G([\id_A]\otimes_{\CC}x)\in KK^H\big(A\otimes C_0(I),
A\otimes C_0(\Om)\big)$$ of  the class $[\id_A]\otimes_{\CC}x\in
KK^G\big(A\otimes C_0(I), A\otimes C_0(\Om)\big)$ induces
isomorphisms
$$[\cdot]\otimes \res_H^G([\id_A]\otimes_{\CC}x): K_*^H\big(A\otimes
C_0(I)\big)\stackrel{\cong}{\to}
K_*^H(A\otimes \C_0(\Om))$$ and, via  descent,
$$[\cdot]\otimes j_H( {\res_H^G} ([\id_A]\otimes_{\CC}x)): K_*\big(\big(A\otimes
C_0(I)\big)\rtimes H\big)\stackrel{\cong}{\to}
K_*\big((A\otimes C_0(\Om))\rtimes H\big).$$

Moreover, if $H\subseteq G$ is compact such that $A\rtimes H_i$ lies
in the bootstrap class for all stabilizers $H_i=\{h\in H: h i=i\}$
(this is for example always true if $A$ is type I) or if
$H=\{e\}$ is the trivial group, then $(A\otimes C_0(I))\rtimes H$
and $(A\otimes C_0(\Om))\rtimes H$ are $KK$-equivalent.
\end{proposition}

\begin{proof} It follows from Lemma \ref{lem-finite} (applied to the class ${[\id_A]}\otimes_\CC x$   constructed as in (E2))
 that it suffices to show that the corresponding classes ${[\id_A]}\otimes_{\CC}x_J\in KK^H(A\otimes C_0(J), A\otimes D_J)$ are invertible
for any $H$-invariant finite subset $J\subseteq I$ such that $\{e_i:i\in J\}\cup\{0\}$ is multiplicatively closed, where $D_J\subseteq C_0(\Om)$ is the subalgebra generated by $\{e_i:i \in J\}$.
But this is the case if for all such $J\subseteq I$ the classes $x_J\in KK^H\big(C_0(J), D_J\big)$ are invertible, which follows from
Remark \ref{rem-invertible} above.

Suppose now that $A\rtimes H_i$ lies in the bootstrap class for all
$i\in I$. Then for each finite $H$-invariant subset $J\subseteq I$
we have
$$(A\otimes C_0(J))\rtimes H\cong \bigoplus_{[i]\in H\backslash J} (A\otimes C_0(H/H_i))\rtimes H\sim_M \bigoplus A\rtimes H_i,$$
where the Morita equivalence on the right hand side follows from
Green's imprimitivity theorem \cite[Theorem 17]{Green}. It follows
that $(A\otimes C_0(J))\rtimes H$ is in the bootstrap class. On the
other hand  if $e_i'=e_i-\bigvee_{e_j<e_i}e_j$ for all $i\in J$ is
as in Lemma \ref{lem-nilpotent}, then the morphism $A\otimes
C_0(J)\to A\otimes D_J$ which sends $a\otimes 1_{\{i\}}$ to
$a\otimes e_i'$ for all $i\in J$ is an $H$-equivariant isomorphism,
and hence $(A\otimes C_0(J))\rtimes H\cong (A\otimes D_J)\rtimes H$.
Since the bootstrap class is closed under inductive limits, it
follows that $(A\otimes C_0(I))\rtimes H$ and $(A\otimes
C_0(\Om))\rtimes H$ both lie in the bootstrap class, and hence
satisfy the UCT. Thus the desired $KK$-equivalence follows from
isomorphism in $K$-theory.

If $H=\{e\}$ is the trivial subgroup of $G$, then it follows from
the above arguments (with $A=\CC$) that $x\in KK(C_0(I), C_0(\Om))$
is a $KK$-equivalence, from which it follows that ${[\id_A]}\otimes_\CC
x$ is a $KK$-equivalence between $A\otimes C_0(I)$ and $A\otimes
C_0(\Om)$.
\end{proof}

\begin{remark}

The assertion on type I algebras $A$ in the above proposition follows from the fact that all
type I C*-algebras lie in the bootstrap class (e.g. see
\cite{Black}) and the fact {(see \cite{Tak})} that crossed
products of type I C*-algebras by compact groups are type I.

{Note that the  proposition implies in particular, 
that for all second countable totally disconnected 
spaces $\Om$ 
and any choice $\mathcal V=\{V_i:i\in I\}$ of a regular basis for the compact open sets in $\Om$
(which  exists by Proposition \ref{prop-exists}) 
the class $x\in KK( C_0(I), C_0(\Om))$ constructed above is a $KK$-equivalence.}
\end{remark}

Combining Proposition \ref{prop-finite} with Proposition \ref{prop-BC} we now get:

\begin{theorem}\label{thm-BC}
Suppose that $G$ {satisfies} the Baum-Connes conjecture with
coefficients in $A\otimes C_0(I)$ and in $A\otimes C_0(\Om)$. Then
$$[\cdot]\otimes j_G({[\id_A]}\otimes_{\CC}x):K_*\big(\big(A\otimes
C_0(I)\big)\rtimes_{\alpha\otimes\mu,r}G\big)\to K_*((A\otimes
C_0(\Om)) \rtimes_{\alpha\otimes\tau,r}G)$$ is an isomorphism.
Moreover, if $G$ satisfies the strong Baum-Connes conjecture and
$A\rtimes H$ lies in the bootstrap class for all compact subgroups
$H$ of $G$ which lie in some stabilizer $G_i$ for the action of $G$
on $I$, or if $G$ satisfies the strong Baum-Connes conjecture and
has no non-trivial compact subgroups, then
$$j_G({[\id_A]}\otimes_{\CC}x)\in KK\big((A\otimes C_0(I))\rtimes_{\alpha\otimes\mu, r}G, (A\otimes C_0(\Om))\rtimes_{\alpha\otimes\tau,r}G\big)$$
is a $KK$-equivalence.
\end{theorem}

\begin{remark}\label{rem-BC} 
 If $I$ is a countable discrete $G$-space, then it follows from the results of \cite[Theorem 2.5]{CE} and
\cite[Proposition 2.6]{CEN} and the decomposition of $A\otimes C_0(I)$ as given  in (\ref{eq-decom}) below, that
$G$ satisfies the Baum-Connes conjecture with coefficients in $A\otimes C_0(I)$ if (and only if) all stabilizers $G_i$ for the action of $G$ on $I$ satisfy the conjecture with coefficients in $A$ with respect to the restriction of the given action of $G$ on $A$ to the subgroups $G_i$.

On the other hand, $G$ satisfies the Baum-Connes conjecture with coefficients in $A\otimes C_0(\Om)$ if and only if the transformation groupoid $\Om\rtimes G$ satisfies the groupoid version of the Baum-Connes conjecture with coefficients in $A\otimes C_0(\Om)$ induced by the given $G$-action on $A$. We refer to \cite{Tu} for the formulation of the Baum-Connes conjecture for groupoids. There it is shown  that $\Om\rtimes G$ satisfies the Baum-Connes conjecture with arbitrary coefficients if the groupoid
$\Om\rtimes G$ is amenable (or, more generally, a-$T$-menable).  This situation is much more general than simply assuming  that $G$ is amenable or  a-$T$-menable.
Thus we see that the conditions on the Baum-Connes conjecture in Theorem \ref{thm-BC} are in particular satisfied for every $G$-algebra $A$ if the following two conditions hold:
\begin{enumerate}
\item All stabilizers $G_i$ for the action of $G$ on $I$ are amenable (or a-$T$-menable), and
\item the transformation groupoid $\Om\rtimes G$  is amenable (or a-$T$-menable).
\end{enumerate}
\end{remark}

Since $I$ is discrete,  we get a decomposition of
$A\otimes C_0(I)$ as a direct sum of $G$-algebras
\begin{equation}\label{eq-decom}
A\otimes C_0(I)\cong \bigoplus_{[i]\in G\backslash I} A\otimes C_0(G\cdot i)\cong \bigoplus_{[i]\in G\backslash I}
A\otimes C_0(G/G_i)
\end{equation}
where $G_i$ denotes the stabilizer $G_i:=\{g\in G: g\cdot i=i\}$ for $i\in I$ under the $G$-action. We therefore get a
decomposition of the reduced crossed products
$$(A\otimes C_0(I))\rtimes_{\alpha\otimes\mu, r}G\cong \bigoplus_{[i]\in G\backslash I} (A\otimes C_0(G/G_i))\rtimes_{\alpha\otimes\mu_i, r}G,$$
where $\mu_i: G\to \Aut(C_0(G/G_i))$ is the action by left
translation. Thus, by {a} version of Green's imprimitivity
theorem (see \cite[Theorem 17]{Green} and \cite{QS})  we have a
canonical Morita equivalence
$$(A\otimes C_0(G/G_i))\rtimes_{\alpha\otimes\mu_i, r}G\sim_M A\rtimes_{\alpha,r}G_i.$$
Therefore, by Morita invariance and continuity of $K$-theory, we obtain a canonical isomorphism
$$\bigoplus_{[i]\in G\backslash I} K_*(A\rtimes_{\alpha,r}G_i)\cong K_*\big((A\otimes C_0(I))\rtimes_{\alpha\otimes\mu, r}G\big).$$
Combining this with the isomorphism of Theorem \ref{thm-BC}, we obtain

\begin{corollary}\label{cor-BC}
Suppose that $G$, $\Om$, $\alpha:G\to\Aut(A)$  and $\cV=\{V_i:i\in I\}$  are as in Theorem \ref{thm-BC}. Then there is a canonical
$KK$-equivalence  $y_I\in KK_0\big(\bigoplus_{[i]\in G\backslash I} A\rtimes_{\alpha,r}G_i, (A\otimes C_0(I))\rtimes_{\alpha\otimes\mu,r}G\big)$. Combined with the isomorphism of Theorem \ref{thm-BC} we get an  isomorphism
$$
\bigoplus_{[i]\in G\backslash I} K_*(A\rtimes_{\alpha,r}G_i)\cong
K_*\big((A\otimes C_0(\Om))\rtimes_{\alpha \otimes {\tau},r}G\big)
$$
If we have $KK$-equivalence in Theorem \ref{thm-BC}, then the above isomorphism is also induced by a $KK$-equivalence.

In particular, in the special case $A=\CC$ we get an isomorphism
$$\bigoplus_{[i]\in G\backslash I} K_*(C_r^*(G_i))\cong
K_*\big(C_0(\Om)\rtimes_{{\tau},r}G\big).$$ If $G$ satisfies the strong
Baum-Connes conjecture, this isomorphism is obtained from a
$KK$-equivalence $\bigoplus_{[i]\in G\backslash I}
C_r^*(G_i)\sim_{KK} C_0(\Om)\rtimes_{{\tau},r}G$.
\end{corollary}

\begin{remark}\label{rem-homomorphism}
We  point out that the isomorphism $$\bigoplus_{[i]\in G\backslash I} K_*(A\rtimes_{\alpha,r}G_i)\cong
K_*\big((A\otimes C_0(\Om))\rtimes_{\alpha\otimes \mu,r}G\big)$$ of the theorem is induced by a $*$-homomorphism
$$\Psi: \bigoplus_{[i]\in G\backslash I}  A\rtimes_{\alpha,r}G_i\to \K(\ell^2(I))\otimes \big((A\otimes C_0(\Om))\rtimes_{\alpha\otimes\mu, r}G\big).$$
which can be described as follows. First of all we have a homomorphism
$$\eta_i:A\rtimes_{\alpha,r}G_i\to (A\otimes C_0(G\cdot i))\rtimes_{\alpha\otimes\tau,r}G$$
given by the inclusion $A\rtimes_{\alpha,r}G_i\hookrightarrow
(A\otimes C_0(G\cdot i))\rtimes_{\alpha\otimes\tau,r}G_i$ induced
from the $G_i$-equivariant inclusion $A\hookrightarrow  A\otimes
C_0(G\cdot i); a\mapsto a\otimes \delta_{i}$ followed by the
inclusion $(A\otimes C_0(G\cdot
i))\rtimes_{\alpha\otimes\tau,r}G_i\hookrightarrow (A\otimes
C_0(G\cdot i))\rtimes_{\alpha\otimes\tau,r}G$ (use \cite[Lemma
2.5.2]{CEL}). This maps $A\rtimes_{\alpha,r}G_i$ bijectively onto a
full corner of $(A\otimes C_0(G\cdot
i))\rtimes_{\alpha\otimes\tau,r}G$ which establishes Green's Morita
equivalence $A\rtimes_{\alpha,r}G_i\sim_M (A\otimes C_0(G\cdot
i))\rtimes_{\alpha\otimes\tau,r}G$. Using the  decomposition
$$(A\otimes C_0(I))\rtimes_{\alpha\otimes\tau,r}G\cong  \bigoplus_{[i]\in G\backslash I} (A\otimes C_0(G\cdot i))\rtimes_{\alpha\otimes\tau,r}G,$$ we then obtain a $*$-homomorphism
$$\eta=\bigoplus_{[i]\in G\backslash I} \eta_i: \bigoplus_{[i]\in G\backslash I}  A\rtimes_{\alpha,r}G_i\to \big(A\otimes C_0(I))\rtimes_{\alpha\otimes\tau,r}G$$
which induces the $KK$-equivalence  $y_I$  of Corollary \ref{cor-BC}.

We then recall from our constructions in Notation \ref{not-prod} that, after passing from $A\otimes C_0(\Om)$ to
the stabilization $A\otimes \K(\ell^2(I))\otimes C_0(\Om)$ (with action on $\K(\ell^2(I))$ given by $\Ad\mu$, where by abuse of notation we let $\mu$ denote the unitary representation of $G$ on $\ell^2(I)$ induced by the given action $\mu$ of $G$ on $I$),  the equivariant $KK$-class
${[\id_A]}\otimes x\in KK^G(A\otimes C_0(I), A\otimes C_0(\Om))\cong KK^G\big(A\otimes C_0(I), A\otimes \K(\ell^2(I))\otimes C_0(\Om)\big))$
is represented by the $G$-equivariant $*$-homomorphism
$\id_A\otimes \varphi: A\otimes C_0(I)\to A\otimes \K(\ell^2(I))\otimes C_0(\Om)$
with $\varphi=\bigoplus_{i\in I}\varphi_i$ as in Notation \ref{not-prod}. Recall that $\varphi_i$ sends the element $a\otimes \delta_i\in A\otimes C_0(I)$ to the element $a\otimes p_i\otimes 1_{V_i}$, where $p_i$ denotes the orthogonal projection onto the  $i$th component of $\ell^2(I)$.
Thus, the descent of ${[\id_A]}\otimes x$ is represented, up to
stabilization,  by the $*$-homomorphism
$$\psi:=(\id_A\otimes\varphi)\rtimes G: (A\otimes C_0(I))\rtimes_{\alpha\otimes \mu,r}G\to (A\otimes \K(\ell^2(I))\otimes C_0(\Om))\rtimes_{\alpha\otimes\Ad\mu\otimes \tau, r}G.$$
Finally note that the algebra $ (A\otimes \K(\ell^2(I))\otimes C_0(\Om))\rtimes_{\alpha\otimes\Ad\mu\otimes \tau, r}G$ is
canonically isomorphic to $\K(\ell^2(I))\otimes \big((A\otimes C_0(\Om))\rtimes_{\alpha\otimes \tau,r}G\big)$
by sending a typical element $(a\otimes k\otimes f) u_g$ of the first algebra to the element
$k\cdot \mu_g\otimes ((a\otimes f) u_g)$ of the second algebra (here we denote by $g\mapsto u_g$  the embedding of $G$ into both crossed products). Taking  compositions, we obtain the desired
$*$-homomorphism $\Psi$.
\end{remark}

We show that restricted to each component
$K_*(A\rtimes_{\alpha,r}G_i)$ the isomorphism of Corollary
\ref{cor-BC} has  a  nicer description. For this observe that for
any given $i\in I$ we have a $G_i$-equivariant embedding
$A\hookrightarrow A\otimes C_0(\Om)$ which sends $a\in A$ to
$a\otimes 1_{V_i}$. This induces an embedding
$A\rtimes_{\alpha,r}G_i\hookrightarrow (A\otimes
C_0(\Om))\rtimes_{\alpha\otimes\mu,r}G_i$. Composing this with the
inclusion $ (A\otimes
C_0(\Om))\rtimes_{\alpha\otimes\mu,r}G_i\hookrightarrow  (A\otimes
C_0(\Om))\rtimes_{\alpha\otimes\mu,r}G$ (use \cite[Lemma
2.5.2]{CEL}) we obtain a canonical inclusion
\begin{equation}\label{eq-typical}
\pi_i:A\rtimes_{\alpha,r}G_i\to  (A\otimes C_0(\Om))\rtimes_{\alpha\otimes\tau,r}G
\end{equation}
given on a typical element $a u_g$, $a\in A, g\in G_i$ by $\pi_i(a u_g)=(a\otimes 1_{V_i})u_g$.

\begin{lemma}\label{lem-KK}
For each $i\in I$ let $j_{i} :A\rtimes_{\alpha,r}G_i\to\bigoplus_{[j]\in G\backslash I} A\rtimes_{\alpha,r}G_j$ denote  the
canonical inclusion and let $\pi_{i}$ be as in (\ref{eq-typical}) above. Then we have
$$[\pi_i]= [j_i] {\otimes y_I \otimes} {j_G({[\id_A]}\otimes_{\CC} x)}\in KK\big(A\rtimes_{\alpha,r}G_i,(A\otimes C_0(\Om))\rtimes_{\alpha\otimes\tau,r}G\big).$$
Therefore, the restriction of the isomorphism $$\bigoplus_{[i]\in G\backslash I} K_*(A\rtimes_{\alpha,r}G_i)\cong
K_*\big((A\otimes C_0(\Om))\rtimes_{\alpha\otimes \mu,r}G\big)$$
 of Corollary \ref{cor-BC} to the summand $K_*(A\rtimes_{\alpha,r}G_i)$ is  given by
 $$(\pi_i)_*:K_*(A\rtimes_{\alpha,r}G_i)\to K_*\big((A\otimes C_0(\Om))\rtimes_{\alpha\otimes \mu,r}G\big).$$
 \end{lemma}
\begin{proof} Let $p_i\in \K(\ell^2(I))$ be the projection on the $i$th component of $\ell^2(I)$ and let
$$\theta_i: (A\otimes C_0(\Om))\rtimes_{\alpha\otimes\tau,r}G\to \K(\ell^2(I))\otimes (A\otimes C_0(\Om))\rtimes_{\alpha\otimes\tau,r}G,$$
$  \theta_i(x)=p_i\otimes x$
be the $KK$-equivalence induced by $p_i$ (note that this $KK$-equivalence does not depend on  the
particular choice of the rank-one projection $p_i\in \K(\ell^2(I))$). By Remark \ref{rem-homomorphism} we have
$$ {y_I \otimes j_G({[\id_A]}\otimes_{\CC} x) \otimes} [\theta_i]=[\Psi]$$
in $KK\big(\bigoplus_{[i]\in G\backslash I} A\rtimes_{\alpha,r}G_i\,,\,
\K(\ell^2(I))\otimes(A\otimes C_0(\Om))\rtimes_{\alpha\otimes\tau,r}G\big)$
with
$$\Psi:\bigoplus_{[i]\in G\backslash I} A\rtimes_{\alpha,r}G_i\to
\K(\ell^2(I))\otimes(A\otimes C_0(\Om))\rtimes_{\alpha\otimes\tau,r}G$$ as in the remark. So it suffices to show
that $\theta_i\circ \pi_i= \Psi\circ j_i$. By construction, on a typical element $a u_g\in A\times_{\alpha,r}G_i$
we have $\theta_i\circ \pi_i(a u_g)=p_i\otimes (a\otimes 1_{V_i})u_g$. On the other hand, following the description
of $\Psi$ in the remark, we get $\Psi(j_i(au_g))=p_i \mu_g\otimes (a\otimes 1_{V_i})u_g$. But remember that at this point, $\mu_g$ is the unitary operator acting on $\ell^2(I)$ via the action of $G$ on $I$. Since $g\in G_i$ lies in the
stabilizer of $i\in I$, we get $p_i\mu_g=p_i$. Thus $\theta_i\circ \pi_i(a u_g)= \Psi\circ j_i(au_g)$ for all $a\in A, g\in G_i$
and the result follows.
\end{proof}

\begin{example}\label{ex-cantor-K}
As a first example we want to study the $K$-theory of the Cantor set $\Om=\{1,-1\}^{\ZZ}$ of Example \ref{ex-cantor-action}, i.e.,
we consider the action of $\ZZ$ on $\Om$ given by the shift $(m\cdot x)_n=x_{n-m}$ for $m\in \ZZ$ and $x=(x_n)_{n\in \ZZ}\in \Om$.
Let $\mathcal F(\ZZ)$ denote the family of finite subsets of $\ZZ$ and let  $\cV=\{V_F: F\in \mathcal F(\ZZ)\}$ be the regular basis for the compact open sets in $\Om$ as constructed in Example \ref{ex-cantor}. Then the corresponding action of $\ZZ$ on $\mathcal F(\ZZ)$ is given by translation. {It is free on $\mathcal F(\ZZ)^*:=\mathcal{F}(\ZZ)\smallsetminus\{\emptyset\}$ and it fixes the empty set. Thus, since $\ZZ$  satisfies the strong Baum-Cones conjecture, Corollary~\ref{cor-BC} shows that for any action $\alpha:\ZZ\to \Aut(A)$, we obtain a
$KK$-equivalence between $\big(A\rtimes_{\alpha}\ZZ\big)\oplus\left( \bigoplus_{[F]\in \ZZ\backslash\mathcal F(\ZZ)^*} A\right)$ and $(A\otimes C(\Om))\rtimes_{\alpha\otimes\tau}\ZZ$. In particular, we obtain an isomorphism
$$K_*(A\rtimes_\alpha\ZZ)\oplus \left(\bigoplus_{[F]\in \ZZ\backslash \mathcal F(\ZZ)^*}K_*(A)\right)\cong K_*\big((A\otimes C(\Om))\rtimes_{\alpha\otimes\tau}\ZZ\big).$$
On the summand $K_*(A)$ corresponding to some $[F]\in \ZZ\backslash \mathcal F(\ZZ)^*$, the isomorphism
is induced by the inclusion $a\mapsto a\otimes 1_{V_F}\in A\otimes C(\Om)\subseteq (A\otimes C(\Om))\rtimes_{\alpha\otimes\tau}\ZZ$. On the summand $A\rtimes_{\alpha}\ZZ$ it is given by the descent of the $\ZZ$-equivariant inclusion 
$a\mapsto a\otimes 1_{\Om}$ of $A$ into $A\otimes C(\Om)$.
In the special case where $A=\CC$ we obtain a $KK$-equivalence between
$C(\mathbb T)\oplus C_0(\ZZ\backslash \mathcal F(\ZZ)^*)$ and $C_0(\Om)\rtimes_{\tau}
\ZZ$.}
\end{example}

We now present examples of actions  for which a $G$-invariant regular basis
of the compact open sets of $\Om$ cannot exist. {Our first example shows that there are indeed many $\ZZ$-actions 
on the Cantor set, which do not allow a $\ZZ$-invariant regular basis
for the compact open sets.}

{\begin{example}\label{ex-notexists}
For any prime $p$ let $\ZZ_p$ denote the ring of $p$-adic integers. The underlying additive group is totally disconnected and compact and, as a space, is homeomorphic to the Cantor set. Moreover,
 $\ZZ$ embeds into $\ZZ_p$ as a dense subgroup via $n\mapsto n\cdot 1_p$, where  $1_p$ denotes the multiplicative unit of $\ZZ_p$. 
Consider the translation action $\tau$ of $\ZZ$ on $\ZZ_p$ given by $\tau_n(x)=x+n1_p$. This is a minimal  action 
of the type as studied by Riedel in \cite{Ried}. Let $\chi :\widehat{\ZZ}_p\to\mathbb T$ denote the character given by
evaluation at $1_p$. Since $1_p$ generates a dense subgroup of $\ZZ_p$, the character $\chi$ is faithful.
The image $G:=\chi(\widehat{\ZZ}_p)$ is the Pr{\"u}fer $p$-group, i.e.,  the union of all 
cyclic subgroups of $\mathbb T$ with order a power $p^m$ of $p$. Let $\widehat{\tau}$ denote 
the translation action  of $\widehat{\ZZ}_p\cong G$ on $\mathbb T$. Then there is a canonical isomorphism 
$$C(\mathbb T)\rtimes_{\widehat\tau} \widehat{\ZZ}_p\cong C(\ZZ_p)\rtimes_\tau \ZZ$$
which can be obtained by  representing the crossed products faithfully on $L^2(\widehat{\ZZ}_p\times \mathbb T)$ and 
$L^2(\ZZ\times \ZZ_p)$ via the  canonical regular representations, respectively, and then check that conjugation with the  Plancherel isomorphism 
$\Psi : L^2(\widehat{\ZZ}_p\times \mathbb T)\to L^2(\ZZ_p\times \ZZ)\cong L^2(\ZZ\times \ZZ_p)$ induces the desired 
isomorphism.
Thus we can apply \cite[Theorem 3.6]{Ried} which implies that $K_0(C_0(\ZZ_p)\rtimes \ZZ)$ is isomorphic 
to the group $\ZZ[\frac{1}{p}]=\{\frac{k}{p^l}: k\in \ZZ, l\in \NN_0\}$ (note that the crossed product in question is also isomorphic to the well known Bunce-Deddens algebra). The abelian group $\ZZ[\frac{1}{p}]$ 
 is not isomorphic to any direct sum of copies of $\ZZ$. 
But if there were a $\ZZ$-invariant  regular basis $\cV$ for $\ZZ_p$, Corollary \ref{cor-BC} would imply that 
$K_0(C_0(\ZZ_p)\rtimes \ZZ)$ is isomorphic  to a direct sum of copies of $\ZZ$. 
\end{example}}

{More examples of the above type can be obtained from the results in \cite{Ried}. For example, the 
odometer actions described in \cite[p. 332]{Put} give a big class of $\ZZ$-actions on the Cantor set for which a $\ZZ$-invariant regular basis for the compact open sets cannot exist. 
The following corollary of Theorem \ref{thm-BC} will be used to give  such  an example with
an action of the free group $F_n$.}

\bcor
Let $G$, $\Omega$ be as in Theorem~\ref{thm-BC} (in particular, we assume that $\Omega$ has a $G$-invariant regular basis). In addition, let $G$ be discrete and $K$-amenable in the sense of \cite{Cun}. Then $K_0(C_0(\Omega) \rtimes_{\tau,r} G)$ contains a copy of $\Zz$ as a direct summand.
\ecor

\begin{proof}
Corollary~\ref{cor-BC} tells us that $\bigoplus_{[i]\in G\backslash I} K_0(C_r^*(G_i))\cong K_0\big(C_0(\Om)\rtimes_{\blue{\tau},r}G\big)$. Now since $G$ is $K$-amenable, each of the subgroups $G_i$ is also $K$-amenable by \cite{Cun}. Thus $K_0(C_r^*(G_i)) \cong K_0(C^*(G_i))$ contains a copy of $K_0(\CC) \cong \Zz$ as a direct summand.
\end{proof}

With the help of this corollary, we can now present more examples for which a $G$-invariant regular basis cannot exist. We thank M. R{\o}rdam who drew our attention to the existence of
such examples.

\begin{example}\label{ex-rordam}
Consider the dynamical system from \cite[\S~8.2]{Li-nuc} with the free group $\Fz_n$ acting on the positive part $(\partial \Fz_n)_+$ of its Gromov boundary by left translations. As observed in \cite[\S~8.2]{Li-nuc}, the crossed product $C_0((\partial \Fz_n)_+) \rtimes_r \Fz_n$ is Morita equivalent to $\cO_n$, so $K_0(C_0((\partial \Fz_n)_+) \rtimes_r \Fz_n) \cong \Zz / (n-1) \Zz$. Thus we conclude using the previous corollary that there cannot exist an $\Fz_n$-invariant regular basis for the compact open subsets of $(\partial \Fz_n)_+$.
\end{example}

\begin{question}
Is there an intrinsic characterization for the existence of such invariant regular bases for the compact open subsets in terms of the underlying topological dynamical system?
\end{question}

\begin{remark}
Even if we cannot find a $G$-invariant regular basis there is always the following regularization procedure:

Let $\Omega$ be a totally disconnected, second countable, locally compact $G$-space and consider the $G$-algebra $D = C_0(\Omega)$. We can always find a generating family of compact open subsets $\cV$ of $\Omega$ such that

\begin{itemize}
\item $\cV \cup \left\{ \emptyset \right\}$ is closed under finite intersections,
\item $\cV$ is $G$-invariant.
\end{itemize}

One possibility would be $\cV=\cU_c(\Om)$.  More generally, we can start with an arbitrary generating family $\cV_0$ and let $\cV$ be the smallest family satisfying the two desired conditions above and containing $\cV_0$. Of course, in general, $\cV$ will not be independent. But we can define $D \left\langle \cV \right\rangle$ as the universal C*-algebra $C^*(\left\{e_V \text{: } V \in \cV \right\} \vline \cR)$ with the set of relations $\cR$ given by:
$$
  e_V = e_V^* = e_V^2 \quad\text{and}\quad
  e_{V_1} e_{V_2} =
 \left\{ \begin{matrix}
  e_{V_1 \cap V_2} & \text{ if } V_1 \cap V_2 \neq \emptyset \\
  0 & \text{ else}
  \end{matrix}\right..
$$

As explained in \cite[\S~2]{Li-nuc}, the family of projections $\left\{e_V \text{: } V \in \cV \right\} \subseteq D \left\langle \cV \right\rangle$ is independent. And by universal property, there is a canonical surjective homomorphism $D \left\langle \cV \right\rangle \to D$ given by $e_V \mapsto 1_V$. Let $D_1$ be the kernel of this surjection. We then obtain a short exact sequence $ 0 \to D_1 \to D \left\langle \cV \right\rangle \to D \to 0$, and $D_1$ will be $\left\{ 0 \right\}$ if and only if the family $\cV$ we started with was already independent.

In addition, by universal property of $D \left\langle \cV \right\rangle$, every $g \in G$ gives rise to an automorphism of $D \left\langle \cV \right\rangle$ which is determined by $e_V \mapsto e_{g V}$. With this $G$-action on $D \left\langle \cV \right\rangle$, the canonical homomorphism $D \left\langle \cV \right\rangle \to D$ becomes $G$-equivariant. Thus if $G$ is exact, we obtain from the exact sequence above the following exact sequence of the reduced crossed products:
$$ 0 \to D_1 \rtimes_r G \to D \left\langle \cV \right\rangle \rtimes_r G \to D \rtimes_r G \to 0.$$
We could also dualize and obtain with $\Omega_1 = \Spec D_1$ and $\Omega \left\langle \cV \right\rangle = \Spec (D \left\langle \cV \right\rangle)$ the following exact sequence:

\begin{equation}
\label{ses}
  0 \to C_0(\Omega_1) \rtimes_r G \to C_0(\Omega \left\langle \cV \right\rangle) \rtimes_r G \to C_0(\Omega) \rtimes_r G \to 0.
\end{equation}

Since $\left\{e_V \text{: } V \in \cV \right\} \subseteq D \left\langle \cV \right\rangle$ is independent, this family of projections corresponds to a regular basis of $\Omega \left\langle \cV \right\rangle$, so that our method of computing K-theory applies to the crossed product in the middle of {\eqref{ses}}. The idea would then be to try to use the six term exact sequence in K-theory for {\eqref{ses}} to compute K-theory for $C_0(\Omega) \rtimes_r G$. This of course means that we have to compute K-theory for the ideal in {\eqref{ses}} first. Since $D_1 = C_0(\Omega_1)$ is again of the same form as $D = C_0(\Omega)$, we could iterate this regularization process. However, the question is whether in this iteration, we will at some point be able to determine K-theory for the kernel, i.e. for the analogue of $D_1 \rtimes_r G$.
\end{remark}

\section{$K$-theory of semigroup crossed products} \label{sec-semigroup}
In this section we want to apply the results of the previous section to the study of the $K$-theory of certain semigroup crossed products.
Throughout this section we assume that  $P\subseteq G$ is a subsemigroup of the group $G$ which contains the unit element
$e\in G$. By a right ideal of $P$ (resp. a right $P$-ideal in $G$), we {mean} a subset $X$ of $P$ (resp. $G$) such that $XP=X$. For an arbitrary subset $X$ of $G$ and for $g\in G$ we write $g\cdot X=\{gx:x\in X\}\subseteq G$ for the translate of $X$ by $g$.
Moreover, if $X\subseteq P$ and $p\in P$ we write $pX:=p\cdot X$ and $p^{-1}X=\{y\in P: py\in X\}=(p^{-1}\cdot X)\cap P$.
It is important to observe the difference between the set $p^{-1}\cdot X\subseteq G$ and the set $p^{-1}X\subseteq P$ defined above! We recall from \cite{Li-nuc} the following  definition of constructible right ideals in $P$ and $G$:

\begin{definition}\label{def-ideals}
Let $P\subseteq G$ be as above. Then the set of {\em constructible right ideals} $\JP$  of $P$ is defined as
 the smallest family of subsets of $P$ which contains the empty set $\emptyset$ as well as $P$ and also $pX, p^{-1}X$ for all $X\in \JP$ and $p\in P$.

The set of {\em constructible right $P$-ideals} $\JG$ in $G$ is the
smallest left translation invariant family of subsets $X\subseteq G$
which contains $\JP$ and which is closed under taking finite
intersections.
\end{definition}

As observed in \cite[\S~3]{Li-am}, $\JP$ is automatically closed under finite intersections.

%The following notion will also be crucial for this paper:
%
%\begin{definition}\label{def-independentset} Let $ \cJ$ be a subset of the power set  $\mathcal P(Y)$ of a set $Y$. We call $\cJ$ {\em independent}, if for every finite selection $X, X_1,\ldots, X_k$ of elements in $\cJ$ such that $X=\cup_{i=1}^kX_i$, there must be an index $i\in \{1,\ldots, k\}$ such that $X_i=X$.
%\end{definition}

If $Y$ is a discrete space and $X\subseteq Y$  we let $E_X:\ell^2(Y)\to \ell^2(X)\subseteq \ell^2(Y)$ denote the orthogonal projection, which is given by multiplication with the characteristic function $ 1_X$ of $X$. If $\cJ\subseteq \mathcal P(Y)$, we let
\begin{equation}\label{eq-D}
D(\cJ)=C^*(\{E_X: X\in \mathcal J\})\subseteq \cL(\ell^2(Y))
\end{equation}
denote the commutative C*-algebra generated by the projections $E_X$, $X\in  \cJ$ and we write $\Om(\mathcal J)$ for
the Gelfand dual $\Spec(D(\mathcal J))$. Recall from Lemma \ref{lem-generators} and Lemma \ref{lem-projections}
that $\Om(\cJ)$ is totally disconnected and that the
family $\cV=\{V_X: X\in \cJ\}$, with $V_X:=\widehat{E_X}^{-1}(\{1\})$, generates the compact open subsets of $\Om(\cJ)$.
Moreover, it is clear that the  representation
$M:\ell^\infty(Y)\to\cL(\ell^2(Y))$ by multiplication operators $M(f)\xi=f\cdot\xi$ restricts to an isomorphism between
$C^*(\{1_X: X\in \cJ\})\subseteq \ell^{\infty}(Y)$ and $D(\cJ)$.
\medskip

If $P\subseteq G$ is a subsemigroup of a group $G$, we put $\DP:=D(\JP)$ and
$\DG:=D(\JG)$ where $D(\JP)$ and $D(\JG)$ are as in (\ref{eq-D}). Similarly, we shall simply write $\Om_P$ and $\Om_{P\subseteq G}$ for {the}
corresponding totally disconnected spaces $\Om(\JP)$ and $\Om(\JG)$, respectively.
Recall that the reduced {left} semigroup C*-algebra $C_{{\lambda}}^*(P)$ is defined as the sub-C*-algebra of $\cL(\ell^2(P))$ which is generated by the isometries $V_p: \ell^2(P)\to \ell^2(P)$ given by $V_p\delta_q=\delta_{pq}$, where $\delta_q$ denotes the Dirac-function at
$q\in P$. For  $X\subseteq P$ let $E_X$ denote the orthogonal projection from $\ell^2(P)$ onto $\ell^2(X)\subseteq \ell^2(P)$ as in the above discussion. Then
$$V_p E_X V_p^*=E_{pX} \quad\text{and}\quad V_p^*E_X V_p=E_{p^{-1}X}.$$
{This shows} that $C_{{\lambda}}^*(P)$ contains all projections $E_X$ with $X\in  \JP$, the set of constructible right ideals in $P$.
Thus, we see that $\DP\cong {C(\Om_P)}$ is a commutative C*-subalgebra of $C_{{\lambda}}^*(P)$.
On the other hand, since the set $\JG$ of constructible right $P$-ideals in $G$ is closed under left translation with elements of $G$,
the C*-algebra $\DG= D(\JG)\subseteq \ell^\infty(G)$ is also invariant under the left translation action $\tau:G\to \Aut(\ell^\infty(G))$.
Thus  we obtain a well defined action $\tau:G\to\Aut({\DG})$ given on the generators by
$$\tau_g(E_X)=E_{{g \cdot X}},\quad \forall X\in \mathcal {\JG}.$$

In what follows we want to compare  $C_{{\lambda}}^*(P)$ with the reduced crossed product ${\DG}\rtimes_{\tau,r}G\cong C_0(\Om_{P\subseteq G})\rtimes_{\tau,r}G$.
Indeed, we want to consider a more general situation in which we start with an action
$\alpha:G\to \Aut(A)$  of $G$ on a C*-algebra $A$.  Then $\alpha$ restricts to an action of $P$ on $A$ by automorphisms, and
we can form the reduced semigroup crossed product $A\rtimes_{\alpha,r}P$ as follows:

Assume that $A$ is represented faithfully and nondegenerately
on the Hilbert space $\cH$. We then obtain a faithful representation $\widetilde{\alpha_P}:A\to \cL(\cH\otimes \ell^2(P))$ by
\begin{equation}\label{eq-tildealpha}
\widetilde{\alpha_P}(a)(\xi\otimes \epsilon_x):= \alpha_x^{-1}(a)\xi\otimes\epsilon_x\quad\forall \xi\in \cH, x\in P.
\end{equation}
The reduced semigroup crossed product $A\rtimes_{\alpha,r}P$ is then defined as
\begin{equation}\label{eq-crossed}
A\rtimes_{\alpha,r}P:=C^*\left(\{\widetilde{\alpha_P}(a)(1_\cH\otimes V_p) : a\in A, p\in P\}\right)\subseteq \cL(\cH\otimes \ell^2(P))
\end{equation}
(see \cite{Li-nuc} for more details).

Let us now recall some results of \cite{Li-nuc} concerning the question under what conditions on $P\subseteq G$ we can realize $A\rtimes_{\alpha,r}P$ as a full corner of the reduced crossed product $\big(A\otimes {\DG}\big)\rtimes_{\alpha\otimes \tau, r}G$. We start by recalling \cite[Lemma 3.6]{Li-nuc}:

\begin{lemma}\label{lem-repcrossed}
Let $\pi:A\otimes {\DG}\to \cL(\cH\otimes \ell^2(G))$ be the representation defined by
$$\pi(a\otimes d)(\xi\otimes \epsilon_x)=\alpha_x^{-1}(a)\xi\otimes d\epsilon_x.$$
Then $(\pi, 1_{\cH}\otimes \lambda_G)$ is a covariant homomorphism of $\big(A\otimes {\DG},\; G, \;\alpha\otimes \tau\big)$
on $\cH\otimes \ell^2(G)$ which induces a faithful representation of the reduced crossed product
$\big(A\otimes {\DG}\big)\rtimes_{\alpha\otimes \tau, r}G$ on $\cL(\cH\otimes \ell^2(G))$.
\end{lemma}

Following the notation of \cite{Li-nuc} we introduce the following

\begin{notation}\label{not-reducedimage}
We let $A\rtimes_{\alpha,r}(P\subseteq G)$ denote the (isomorphic) image of $\big(A\otimes {\DG}\big)\rtimes_{\alpha\otimes \tau, r}G$ in $\cL(\cH\otimes \ell^2(G))$ under the representation $\pi\rtimes(1_{\cH}\otimes\lambda_G)$ of the above lemma.
\end{notation}

Since $P\in \JP\subseteq  {\JG}$ we have
$1_A\otimes E_P\in M(A\otimes {\DG})$ which embeds canonically into the multiplier algebra of $(A\otimes {\DG})\rtimes_{\alpha\otimes \tau, r}G$.
Extending the representation $\pi\rtimes(1\otimes\lambda_G)$ of Lemma \ref{lem-repcrossed} to the multiplier algebra maps
$1_A\otimes E_P$ to the projection $1_{\cH} \otimes E_P\in \cL(\cH\otimes\ell^2(G))$. We therefore may consider the
corner
$$(1_{\cH} \otimes E_P)\big(A\rtimes_{\alpha,r}(P\subseteq G)\big)(1_{\cH}\otimes E_P)\subseteq \cL(\cH\otimes \ell^2(P))$$
inside $A\rtimes_{\alpha,r}(P\subseteq G)$.

The following important  lemma is the combination of  \cite[Lemma 3.8]{Li-nuc} with
\cite[Lemma 3.9]{Li-nuc}:

\begin{lemma}\label{lem-toeplitz} Let $P\subseteq G$ be a subsemigroup of the group $G$. Then for every system $(A,\, G,\,\alpha)$ we have that
$(1_{\cH} \otimes E_P)\big(A\rtimes_{\alpha,r}(P\subseteq G)\big)(1_{\cH}\otimes E_P)$ is a full corner
of $A\rtimes_{\alpha,r}(P\subseteq G)$ which contains $A\rtimes_{\alpha,r}P$, and the following are equivalent:
\begin{enumerate}
\item $A\rtimes_{\alpha,r}P=(1_{\cH} \otimes E_P)\big(A\rtimes_{\alpha,r}(P\subseteq G)\big)(1_{\cH}\otimes E_P)$ for every C*-dynamical system $(A,\; G,\; \alpha)$,
\item $C_{{\lambda}}^*(P)=E_PC_r^*(P\subseteq G) E_P$, where we set {$C_r^*(P\subseteq G) :=\CC\rtimes_{\id, r}(P\subseteq G)\cong D_{P\subseteq G}\rtimes_{\tau,r}G$},
\item For all $g\in G$ we have $E_P\lambda_g E_P\in C_{{\lambda}}^*(P)$,
\end{enumerate}
and either of these statements implies $E_P{\DG} E_P\subseteq \DP$.
\end{lemma}

We now recall the definition of the Toeplitz condition from \cite{Li-nuc}:

\begin{definition}[{cf. \cite[Lemma 3.9 and Definition 4.1]{Li-nuc}}]\label{def-toeplitz} Let $P\subseteq G$ be a subsemigroup of the group $G$. We say that
\begin{enumerate}
\item $P\subseteq G$ satisfies the {\em Toeplitz condition} if for all $g\in G$ with $E_P\lambda_gE_P \neq 0$, there exist $p_i$,$q_i$ in $P$ such that $E_P\lambda_gE_P=V_{p_1}^*V_{q_1}\cdots V_{p_n}^*V_{q_n}$,
\item $P\subseteq G$ satisfies the {\em weak Toeplitz condition} if the equivalent conditions (1), (2) and (3) of Lemma \ref{lem-toeplitz} are satisfied, and
\item $P\subseteq G$ satisfies the {\em $K$-theoretic Toeplitz condition} if for every system $(A,\,G,\,\alpha)$ the inclusion map
$\iota: A\rtimes_{\alpha,r}P\to (1_{\cH} \otimes E_P)\big(A\rtimes_{\alpha,r}(P\subseteq G)\big)(1_{\cH}\otimes E_P)$ induces
an isomorphism of $K$-theory groups.
\end{enumerate}
\end{definition}

It is clear from condition (3) of Lemma \ref{lem-toeplitz} that the Toeplitz condition implies the weak Toeplitz condition
and it is clear from condition (1) of Lemma \ref{lem-toeplitz} that the weak Toeplitz condition implies the $K$-theoretic Toeplitz condition.
The following result of \cite{Li-nuc} turns out to be extremely useful

\begin{lemma}[{cf \cite[Lemma 4.2]{Li-nuc}}]\label{lem-strongtoeplitz}
Let $P\subseteq G$ such that {the} set $\JP$ of constructible right {ideals} in $P$ is independent in the sense of Definition
\ref{def-independentset} above and assume that $P\subseteq G$ satisfies the Toeplitz condition. Then the following are true:
\begin{enumerate}
\item The set ${\JG}$ of constructible right $P$-ideals in $G$ is independent.
\item For all $g\in G$ and $X\in \JP$ we have $g\cdot X\cap P\in \JP$.
\item $\JG=\{g\cdot X: g\in G, X\in  \JP\}$.
\end{enumerate}
\end{lemma}

Since the set $\JG$ of constructible right $P$-ideals in $G$ is closed under finite intersections, it follows that the
set of projections $\{E_X: X\in \JG\}$ is closed under multiplication. Moreover, since $\DG$ is generated by this set of projections,
it follows that $\{E_X: X\in \JG\}\smallsetminus\{0\}$ forms a regular basis for $\DG\cong C_0(\Om_{P\subseteq G})$ if and only if
$\JG$  is independent in the sense of Definition \ref{def-independentset} (which implies that $\{E_X: X\in \JG\}\smallsetminus\{0\}$
is independent in the sense of
Definition \ref{def-independentproj}).
Thus, if this is satisfied, we are  precisely in the situation of Theorem \ref{thm-BC} (which we may apply to the totally disconnected space $\Om_{P\subseteq G}$ and the regular basis $\cV=\{V_X: X\in \JG\}\smallsetminus\{\emptyset\}$ for the compact open sets of $\Om_{P\subseteq G}$, with $V_X=\widehat{E_X}^{-1}(\{1\})$ for $X\in \JG$ ).
As a consequence we get

\begin{theorem}\label{thm-semiK}
Let $\IG:=\JG\smallsetminus \{\emptyset\}$, let $\alpha:G\to\Aut(A)$ be an action of a countable group $G$ on
a separable C*-algebra $A$ and assume that the following conditions are satisfied for $P\subseteq G$ and $A$:
\begin{enumerate}
\item $P\subseteq G$ satisfies the $K$-theoretic Toeplitz condition;
\item The set $\JG$ of constructible right $P$-ideals in $G$ is independent;
\item $G$ satisfies the Baum-Connes conjecture with coefficients in $A\otimes C_0({\IG})$ and
$A\otimes {\DG}$.
\end{enumerate}
Then there is a canonical isomorphism
\begin{equation}\label{eq-iso1}
\bigoplus_{[X]\in G\backslash {\IG}} K_*(A\rtimes_{\alpha,r}G_X)\cong K_*(A\rtimes_{\alpha,r}P).
\end{equation}
\end{theorem}
\begin{proof} Conditions (2) and (3) imply that Corollary \ref{cor-BC} applies to the regular basis of projections
$\{E_X: X\in {\IG}\}$ and to the commutative C*-algebra ${\DG}\cong C_0(\Om_{P\subseteq G})$ generated by this set.
Thus the corollary gives a canonical isomorphism
\begin{equation*}
\bigoplus_{[X]\in G\backslash {\IG}} K_*(A\rtimes_{\alpha,r}G_X)\cong
K_*\big((A\otimes {\DG})\rtimes_{\alpha\otimes\tau, r}G\big)\cong K_*(A\rtimes_{\alpha,r}(P\subseteq G)),
\end{equation*}
where the second isomorphism follows from
Lemma \ref{lem-repcrossed}. Since $P\subseteq G$ satisfies the $K$-theoretic Toeplitz condition, we further have
$$K_*(A\rtimes_{\alpha,r}P)\cong K_*\big((1_\cH\otimes E_P)(A\rtimes_{\alpha,r}(P\subseteq G))(1_{\cH}\otimes E_P)\big)
\cong K_*\big(A\rtimes_{\alpha, r}(P\subseteq G)\big),$$
where the second isomorphism follows from the fact that $1_\cH\otimes E_P$ is a full projection
in $M(A\rtimes_{\alpha,r}(P\subseteq G))$. \end{proof}

\begin{remark}\label{rem-semiK} Using  Lemma \ref{lem-strongtoeplitz} we see that conditions (1) and (2) in Theorem \ref{thm-semiK} can be replaced by the following (stronger) conditions
\begin{enumerate}
\item[(1')] $P\subseteq G$ satisfies the Toeplitz condition, and
\item[(2')] the set $\JP$ of constructible right ideals in $P$ is independent.
\end{enumerate}
It is often easier to check these conditions rather than conditions (1) and (2) of the theorem.

We should also remark that if $\JG$ is independent and $P\subseteq G$ satisfies the Toeplitz condition, then
$(A\otimes D_{P \subseteq G})\rtimes_{\alpha\otimes\tau,r}G\cong A\rtimes_{\alpha,r}(P\subseteq G)$ is Morita equivalent, and hence
$KK$-equivalent to $A\rtimes_{\alpha,r}P$. Thus if $G$ satisfies the strong Baum-Connes conjecture and {if $A\rtimes H$ lies in the bootstrap class for every finite subgroup $H$ of $G$ which stabilizes some ideal $X\in \IG$ or if}  $G$ satisfies the strong Baum-Connes conjecture and has no non-trivial finite subgroups, then we even get a $KK$-equivalence
$$\bigoplus_{[X]\in G\backslash {\IG}} A\rtimes_{\alpha,r}G_X \sim_{KK} A\rtimes_{\alpha,r}P.$$
\end{remark}

In case of trivial coefficients $A=\CC$, we obtain the following
\begin{corollary}\label{cor-group} Assume that $P\subseteq G$ satisfies  conditions (1), (2) and (3) of Theorem \ref{thm-semiK} for $A=\CC$.
Then there is a canonical isomorphism
$$\bigoplus_{[X]\in G\backslash {\IG}} K_*(C_r^*(G_X))\cong K_*(C_{{\lambda}}^*(P)).$$
If, moreover, $G$ satisfies the strong Baum-Connes conjecture,  this isomorphism is induced by a $KK$-equivalence.
\end{corollary}

\begin{remark}\label{rem-Ore}
 Recall that a semigroup $P$ satisfies the left Ore
condition if and only if it can be imbedded as a subsemigroup of a
group $G$ such that $G=P^{-1}P$.
It follows directly from this
condition that the inclusion $P\subseteq G$ satisfies the
Toeplitz condition. Therefore, if the set $\JP$ of constructible
right deals in $P$ is independent, the same holds for $\JG$ by Lemma
\ref{lem-strongtoeplitz}. Thus, if in addition $G$ satisfies the
Baum-Connes conjecture for suitable coefficients, the results of the
previous section will apply.

This was the situation studied in \cite{CEL} in which we gave a proof of the
above corollary  in this situation together with a large number of interesting applications.
The results obtained here also  allow the study of crossed products by left Ore semigroups
by automorphic actions.
\end{remark}

Interesting examples of left Ore semigroups
 are given by semigroups attached to Dedekind
domains $R$. Let $R^\times:=R\setminus\{0\}$ be its multiplicative semigroup and  let  $R^*$ denote the group of units in $R$. Consider the semigroups $R^\times$, $R^\times/R^*$ and $R\rtimes R^\times$ as studied in detail in \cite{CEL}.
Let $Q(R)$ denote the quotient field of $R$ and  let $Cl_{Q(R)}$ denote the {\em ideal class group} of $Q(R)$. For each $\gamma\in  Cl_{Q(R)}$ we let $I_\gamma\subseteq Q(R)$ be a representative for $\gamma$
(see \cite[\S 8]{CEL} for a more detailed discussion).

\begin{theorem}\label{thm-Dedekind}
 Let $R$ be a Dedekind domain. Then the following are true:
 \begin{enumerate}
 \item For every action $\alpha:R^\times \to \Aut(A)$ there is a canonical isomorphism
 $$ K_*(A\rtimes_{\alpha,r} R^\times)\cong \bigoplus_{\gamma\in Cl_{Q(R)}}K_*(A\rtimes_{\alpha, r} R^*).$$
 \item For every action $\alpha: R^\times/R^*\to \Aut(A)$  there is a canonical isomorphism
 $$ K_*(A\rtimes_{\alpha,r} (R^\times/R^*))\cong \bigoplus_{\gamma\in Cl_{Q(R)}}K_*(A).$$
 \item For every action $\alpha:R\rtimes R^\times\to \Aut(A)$ there is a canonical isomorphism
 $$ K_*\big(A\rtimes_{\alpha,r}(R\rtimes R^\times)\big)\cong \bigoplus_{\gamma\in  Cl_{Q(R)}}K_*\big(A\rtimes_{\alpha,r}(I_\gamma\rtimes R^*)\big).$$
 \end{enumerate}
 \end{theorem}
All  computations necessary for deducing this theorem from Theorem \ref{thm-semiK}  have been done in \cite[\S 8]{CEL}.
Note that in all cases of the above theorem, the enveloping groups are
amenable, hence satisfy the strong Baum-Connes conjecture. Thus whenever
$A$ is type I, the isomorphisms in the above theorem are induced by
$KK$-equivalences.

\section{The case of principal constructible ideals and quasi-lattice ordered groups}\label{sec-quasi-lattices}
In this section we discuss a situation which is particularly nice for our purposes. Assume that $P$ is a subsemigroup of a group $G$ such that all constructible right $P$-ideals in $G$ are principal, i.e. $\JG = \left\{ g \cdot P \text{: } g \in G \right\} \cup \left\{ \emptyset \right\}$. As observed in \cite[\S~8.1]{Li-nuc},it follows that $P \subseteq G$ is Toeplitz. Moreover, another consequence is that $\JP= \left\{ pP \text{: } p \in P \right\} \cup \left\{ \emptyset \right\}$ so that $\JP$ is clearly independent. Conversely, if all constructible ideals of $P$ are principal, i.e. if $\JP = \left\{ pP \text{: } p \in P \right\} \cup \left\{ \emptyset \right\}$, and if $P \subseteq G$ is Toeplitz, then $\JG = \left\{ g \cdot P \text{: } g \in G \right\} \cup \left\{ \emptyset \right\}$. This is a consequence of \cite[Lemma~4.2]{Li-nuc}. Therefore, we may apply our general K-theoretic result to this situation. Since the stabilizer $G_P$ at $P\in \IG$ is equal to the group $P^*$ of invertible elements in $P$, we see that the left hand side of the isomorphism (\ref{eq-iso1}) equals $K_*(A\rtimes_{\alpha,r}P^*)$.

Recall from (\ref{eq-crossed}) the construction of the crossed product $A\rtimes_{\alpha,r}P$.
Let $\iota_A=\widetilde{\alpha_P}:A\to A\rtimes_{\alpha,r}P$ be as in (\ref{eq-tildealpha}) and let
$\iota_{P^*}:P^*\to \U(\ell^2(P))$ given by $\iota_{P^*}(p)=V_p$, where we recall
that for all $p\in P$ we have $V_p\epsilon_x=\epsilon_{px}$, $x\in P$. Then $(\iota_A,\iota_{P^*})$ is covariant
for $(A,P^*,\alpha)$ and we
 obtain a corresponding homomorphism $\iota_A\rtimes\iota_{P^*}:A\rtimes_{\alpha}P^*\to A\rtimes_{\alpha,r}P\subseteq \cB(\cH\otimes \ell^2(P))$.
Now if we decompose $\ell^2(P)$ into the direct sum
$\bigoplus_{[x]\in P^*\backslash P} \ell^2(P^*x)$, and if we identify
 $\ell^2(P^*)$ with $\ell^2(P^*x)$ via the bijection $p\mapsto px$,
 we may check  that  $\iota_A\rtimes\iota_{P^*}$ decomposes into
a multiple of the left regular representation of $A\rtimes_{\alpha,r}P^*$ on the Hilbert space
 $\cB(\cH\otimes \ell^2(P^*))$. Thus $\iota_A\rtimes \iota_{P^*}$ factors through a faithful $*$-homomorphism
 \begin{equation}\label{eq-faithful}
 \iota_{A\rtimes_{r}P^*}:A\rtimes_{\alpha,r}P^*\hookrightarrow  A\rtimes_{\alpha,r}P.
 \end{equation}

\begin{theorem}\label{principal-ideals}
Suppose that $\JG = \left\{ g \cdot P \text{: } g \in G \right\} \cup \left\{ \emptyset \right\}$. Let $G$ act on a C*-algebra $A$ by $\alpha$ such that $G$ satisfies the Baum-Connes conjecture for $A \otimes \DG$ with respect to the diagonal action and that the group of invertible elements $P^*$ in $P$ satisfies the Baum-Connes conjecture for $A$. Then the homomorphism $\iota_{A\rtimes_rP^*}:A\rtimes_{\alpha,r}P^*\to A\rtimes_{\alpha,r}P$ induces an isomorphism
in $K$-theory
$$K_*(A \rtimes_{\alpha,r} P^*)\cong  K_*(A\rtimes_{\alpha,r}P).$$
If, moreover,  $A\rtimes_{\alpha,r}H$ is in the bootstrap class for every finite subgroup of $P^*$ and if
$G$ satisfies the strong Baum-Connes conjecture, then
$\iota_{A\rtimes_r P^*}$ is a $KK$-equivalence.
\end{theorem}
\begin{proof}
Note first that $g\mapsto g \cdot P$ induces a bijection $G/P^*\cong \IG$ and hence it follows from
\cite[Theorem 2.6]{CE}  that $G$ satisfies the Baum-Connes conjecture for $A\otimes C_0(\IG)$ if and only if $P^*$ satisfies the conjecture for $A$. It follows  that the conditions of Theorem \ref{thm-semiK} are satisfied and
that the left hand side of the isomorphism (\ref{eq-iso1}) equals $K_*(A\rtimes_{\alpha,r}P^*)$.
Hence Theorem \ref{thm-semiK} implies the desired result as soon as we have checked that the resulting  isomorphism
(or $KK$-equivalence)
$$\Phi: K_*(A\rtimes_{\alpha,r}P^*)\stackrel{\cong}{\longrightarrow} K_*(A\rtimes_{\alpha,r}P)$$
 is implemented by the inclusion  $\iota_{A\rtimes_rP^*}:A\rtimes_{\alpha,r}P^*\to A\rtimes_{\alpha,r}P$.

For this recall that by Lemma \ref{lem-KK} the isomorphism
$\Phi$ is obtained by the $KK$-class $[\pi_P] {\otimes} [\mu]^{-1}\in KK_0(A\rtimes_{\alpha,r}P^*, A\rtimes_{\alpha,r}P)$
 with
$\pi_P: A\rtimes_{\alpha,r}P^*\to (A\otimes \DG)\rtimes_{\alpha\otimes\tau,r}G$ given by
$\pi_P(au_g)=(a\otimes E_P)u_g$
and where
$\mu: A\rtimes_{\alpha,r}P\to A\rtimes_{\alpha,r}(P\subseteq G)\cong (A\otimes\DG)\rtimes_{\alpha\otimes\tau,r}G$
denotes the realization of $A\rtimes_{\alpha,r}P$ as the full corner
$$(1_{\cH}\otimes E_P)\big(A\rtimes_{\alpha,r}(P\subseteq G)\big) (1_{\cH}\otimes E_P)
\subseteq A\rtimes_{\alpha,r}(P\subseteq G)\cong
(A\otimes \DG)\rtimes_{\alpha\otimes\tau,r}G.$$
Consider the diagram
$$
\begin{CD}
A\rtimes_{\alpha,r}P^* @>\pi_P>>(A\otimes \DG)\rtimes_{\alpha\otimes\tau,r}G\\
@V\iota_{A\rtimes_{\alpha,r}P^*} VV    @V\cong V  \pi\rtimes(1\otimes\lambda) V\\
A\rtimes_{\alpha,r}P   @>>\mu >  A\rtimes_{\alpha,r}(P\subseteq G)
\end{CD}
$$
where the isomorphism $\pi\rtimes(1\otimes\lambda)$ in the right vertical arrow is described in Lemma \ref{lem-repcrossed}. We need to show that this diagram commutes.
Following the definitions  we see that
$$\iota_{A\rtimes_{\alpha,r}P^*}(au_g)=\widetilde{\alpha_P}(a)(1\otimes V_g)
\in \B(\cH\otimes \ell^2(P))$$
with notations as in (\ref{eq-tildealpha}) and (\ref{eq-crossed}). On the other side we
have
$$\pi\rtimes(1_{\cH}\otimes\lambda)\big(\pi_P(a u_g)\big)=
\pi\rtimes(1_{\cH}\otimes\lambda)\big(a\otimes E_P)u_g\big)
=\pi(a\otimes E_P)(1_{\cH}\otimes\lambda_g).$$
By Lemma \ref{lem-repcrossed} we get for  $\xi\in \cH$ and $\epsilon_x\in \ell^2(G)$:
\bglnoz
\pi(a\otimes E_P)(1_{\cH}\otimes\lambda_g)(\xi\otimes \epsilon_x)&=&\alpha_{(gx)^{-1}}(a)\xi\otimes E_P\epsilon_{gx} \\
&=&\left\{\begin{matrix} 0& \text{if $x\notin P$}\\
\widetilde{\alpha_P}(a)(1_{\cH}\otimes V_g)(\xi\otimes \epsilon_x)&\text{if $x\in P$}\end{matrix}\right.
\eglnoz
which gives the desired result.
\end{proof}

As a special case, we can treat quasi-lattice ordered groups. Recall from \cite{Ni1} that $P\subseteq G$ is called {\em quasi-lattice ordered} if the following conditions are satisfied:
\begin{itemize}
\item[{(QL0)}] $P\cap P^{-1}=\{e\}$;
\item[{(QL1)}] for all $g\in G$ the intersection $P\cap (g\cdot P)$ is either empty or of the form $pP$ for some $p\in P$.
\end{itemize}
Condition (QL2) from \cite{Ni1} is automatically satisfied as was observed in \cite{Cr-La}. (QL1) implies that $\JG = \left\{ g \cdot P \text{: } g \in G \right\} \cup \left\{ \emptyset \right\}$.
So we are in the situation that all constructible right $P$-ideals in $G$ are principal. 
 {The Toeplitz condition is shown in \cite[\S 8.1]{Li-nuc}. Hence, since  (QL0) implies  $P^*=\{1\}$ we obtain from Theorem \ref{principal-ideals}: }

\begin{theorem}\label{thm-quasi-lattice}
Suppose that $P\subseteq G$ is quasi-lattice ordered as defined
above. Let $\alpha: G \to \Aut(A)$ be a $G$-action on a C*-algebra
$A$ such that $G$ satisfies the Baum-Connes conjecture for $A\otimes
\DG$ with respect to the diagonal action. Then the canonical
inclusion $\iota_A:A\hookrightarrow A\rtimes_{\alpha,r}P$ induces an
isomorphism
$$K_*(A)\cong  K_*(A\rtimes_{\alpha,r}P).$$
If, moreover, $G$ satisfies the strong Baum-Connes conjecture and if $G$ is torsion-free or $A$ lies in the bootstrap class, then $\iota_A$ is a  $KK$-equivalence.
\end{theorem}

\begin{remark}\label{rem-Toeplitz} The easiest example of a quasi-lattice semigroup is the case $\NN\subseteq \ZZ$
where $\NN$ denotes the non-negative integers.   If $\alpha:\ZZ\to \Aut(A)$ is an action by automorphisms, then the
crossed product $A\rtimes_{\alpha,r}\NN$ coincides with the Toeplitz algebra $\mathcal T=\mathcal T(A)$ as constructed by Pimsner and Voiculescu in \cite{PV}. Indeed, the main work in proving the famous six-term sequence for computing the $K$-theory of $A\rtimes_{\alpha}\ZZ$ as given in \cite[Theorem 2.4]{PV} is to show that the canonical  imbedding $\iota_A:A\to A\rtimes_{\alpha,r}\NN=\mathcal T(A)$ induces an isomorphism in $K$-theory. The above theorem gives a very general version of this important result of Pimsner and Voiculescu.

We should also point out that for positive cones $P$ in certain quasi-lattice ordered groups $G$ (right-angled Artin groups of a special type) and for the trivial coefficient $A=\CC$, the result $K_*(\CC)\cong K_*(C_\lambda^*(P))$ was already obtained in \cite[Theorem~3.3 and Proposition~3.4]{Ivanov}.
\end{remark}

We proceed by constructing natural examples of subsemigroups of groups which satisfy $\JG = \left\{ g \cdot P \;\text{: } g \in G \right\} \cup \left\{ \emptyset \right\}$ without being quasi-lattice ordered:

Let $R$ be a principal ideal domain and $M_n^{\times}(R) := \left\{p \in M_n(R) \;\text{: } \det(p) \neq 0 \right\}$.

\blemma
The constructible right ideals of $P = M_n^{\times}(R)$ are principal.
\elemma
\begin{proof}
We want to show that $\cJ_P = \left\{pP \text{: } p \in P \right\} \cup \left\{\emptyset\right\}$. The only thing which we have to prove is that for every $p, q \in P$, the right ideal $q^{-1} p P$ is also of the form $r P$ for some $r \in P$.

Let $\ti{q} \in P$ satisfy $q \ti{q} = \ti{q} q = \det(q) \cdot 1_n$ ($1_n$ is the identity matrix). Then $q^{-1} p P = (\ti{q} q)^{-1} (\ti{q}p) P = (\det(q) \cdot 1_n)^{-1} (\ti{q}p) P = (\det(q) \cdot 1_n)^{-1} ((\ti{q}p P) \cap \det(q) \cdot P)$. Now consider the Smith normal form of $\ti{q}p$, i.e. find $u$, $v$ in $SL_n(R) \subseteq P$ such that $u \ti{q} p v$ is diagonal, $u \ti{q} p v = \text{diag}(\alpha_1, \dotsc, \alpha_n)$. Thus
\bglnoz
  (\ti{q}p P) \cap \det(q) \cdot P &=& (u^{-1} \text{diag}(\alpha_1, \dotsc, \alpha_n) v^{-1} P) \cap (\det(q) \cdot P) \\
  &=& u^{-1} (\text{diag}(\alpha_1, \dotsc, \alpha_n) P \cap \det(q) \cdot P) \\
  &=& u^{-1} \text{diag}(\lcm(\alpha_1,\det(q)), \dotsc, \lcm(\alpha_n,\det(q))) P.
\eglnoz
Therefore $q^{-1} p P$ can be written as
\bglnoz
  && (\det(q) \cdot 1_n)^{-1} ((\ti{q}p P) \cap \det(q) \cdot P) \\
  &=& (\det(q) \cdot 1_n)^{-1} u^{-1} \text{diag}(\lcm(\alpha_1,\det(q)), \dotsc, \lcm(\alpha_n,\det(q))) P \\
  &=& u^{-1} \text{diag}(\det(q)^{-1} \lcm(\alpha_1,\det(q)), \dotsc, \det(q)^{-1} \lcm(\alpha_n,\det(q))) P.
\eglnoz
Set $r := u^{-1} \text{diag}(\det(q)^{-1} \lcm(\alpha_1,\det(q)), \dotsc, \det(q)^{-1} \lcm(\alpha_n,\det(q)))$, and we arrive at $q^{-1} p P = r P$.
\end{proof}

Going through the proof, it becomes clear that our argument applies whenever $P$ is a subsemigroup of $M_n^{\times}(R)$ such that
\begin{itemize}
\item $SL_n(R) \subseteq P$,
\item for every $q$ in $P$, the element $\ti{q} \in M_n^{\times}(R)$ uniquely determined by $q \ti{q} = \ti{q} q = \det(q) \cdot 1_n$ also lies in $P$,
\item whenever a diagonal matrix $\text{diag}(\alpha_1, \dotsc, \alpha_n)$ lies in $P$, then for every $q$ in $P$, the diagonal matrix $$\text{diag}(\det(q)^{-1} \lcm(\alpha_1,\det(q)), \dotsc, \det(q)^{-1} \lcm(\alpha_n,\det(q)))$$ also lies in $P$.
\end{itemize}

The second condition implies that $P$ is left Ore. Thus $P \subseteq P^{-1} P =: G$ is Toeplitz. As we have shown that all constructible ideals of $P$ are principal, it follows from our discussions above that $\JG = \left\{ g \cdot P \text{ : } g \in G \right\} \cup \left\{ \emptyset \right\}$.

In general, for such semigroups, our conditions concerning the Baum-Connes conjecture are very difficult to verify. But at least for $n=2$ we  get:

\begin{theorem}\label{thm-Mn}
{Suppose that $R$ is a principal ideal domain with field of fractions $K$ and
let $P\subseteq M_2^{\times}(R)$ be a subsemigroup satisfying the above conditions. Let $G=P^{-1}P\subseteq \GL_2(K)$. Then for every action $\alpha:G\to \Aut(A)$ the inclusion $\iota_{A\rtimes_{\alpha,r}P^*}:A\rtimes_{\alpha,r}P^*\to A\rtimes_{\alpha,r}P$ induces an isomorphism
$$K_*(A\rtimes_{\alpha,r}P^*)\cong K_*(A\rtimes_{\alpha,r}P).$$
Moreover, if $A\rtimes_\alpha F$ satisfies the UCT for every finite subgroup $F$ of $P^*$ (which is   true if  $A$ is type I), then $\iota_{A\rtimes_{\alpha,r}P^*}$ induces a $KK$-equivalence.}
\end{theorem}
\begin{proof} {Since $G$ is a countable subgroup of $\GL(2,K)$ it  
 is a-$T$-menable by \cite[Theorem 4]{GHW}. Thus it follows from \cite{HK} that
 $G$ satisfies the strong Baum-Connes conjecture
and the proof follows from Theorem \ref{principal-ideals} and Remark \ref{rem-semiK}.}
\end{proof}

\section{The left and right regular C*-algebra for a semidirect product}\label{sec-left-right}

Let $S$ be a cancellative semigroup. In this section we are interested not only
in the left regular C*-algebra $C^*_\lambda (S)$, but also in the
right regular C*-algebra $C^*_\rho (S)$ generated by the right
regular representation $\rho$ of $S$ on $\ell^2(S)$. Since $C^*_\rho
(S)$ is obviously isomorphic to the left regular C*-algebra of the
opposite semigroup, we might formulate the corresponding arguments
in terms of the left regular representation of the opposite
semigroup. However it will be more convenient to work directly with
the right regular representation. We will be especially interested
in comparing the $K$-theory for the right and left regular
C*-algebras.\mn

\subsection{Ideal independence and Toeplitz condition for the right regular C*-algebra of a semidirect product semigroup}

Assume that the semigroup $S$ is embedded into a group $\bar{S}$ and
let $E$ denote the orthogonal projection of $\ell^2(\bar{S})$ onto
$\ell^2(S)$. Denote by $\rho,\bar{\rho}$ the right regular
representations of $S$ and $\bar{S}$, respectively. We say that
$S\subseteq \bar{S}$ satisfies the right Toeplitz condition if every
operator $E\bar{\rho}(t)E {\neq 0}$ with $t\in \bar{S}$ can be written as a
finite product of elements of the form $\rho (s)$, $s\in S$ and
their adjoints. Of course, this is just saying that the embedding of
the opposite semigroups $S^{\,op}\subseteq \bar{S}^{\,op}$ satisfies
the ordinary Toeplitz condition. By a constructible left ideal in
$S$ we mean a left ideal $I$ such that the opposite ideal $I^{\,op}$
is a constructible right ideal in $S^{\,op}$.

Let now $P$ be a semigroup with unit which acts (on the left) by
injective endomorphisms on the group $H$. We can form the semidirect
product $S=H\rtimes P$. The elements of $H\rtimes P$ are pairs
$(h,p)$, $h\in H,p\in P$ and the multiplication rule is given by
$(h_1,p_1)(h_2,p_2)=(h_1 p_1(h_2),p_1p_2)$. Note that $H\rtimes P$
is left or right cancellative if and only if $P$ is.

\begin{sproposition}\label{lef} The left ideals in $S$ are exactly the subsets of
the form $H\times I$ where $I$ is a left ideal in $P$. The
constructible left ideals in $S$ are exactly those ideals $H\times
I$ where $I$ is a constructible left ideal in $P$.
\end{sproposition}

\begin{proof} The subsets of the given form are obviously left ideals.
Conversely, assume that $K$ is a left ideal in $S$. Then $(x,q)\in
K$ implies that $(H,q)\subseteq K$ and $(y,Pq)\subseteq K$ for all
$y\in H$. Thus $K$ is as claimed.

Moreover, if $K=H\times I$ is a left ideal and $(h,p)\in H\rtimes
P$, then $K(h,p)=H\times Ip$ and $K(h,p)^{-1}=H\times Ip^{-1}$. This
shows the assertion concerning the constructible left ideals.
\end{proof}

\begin{scorollary}\label{C1} If the  constructible left ideals of $P$ form an
independent family, then the same is true for the constructible left
ideals of $S$.\end{scorollary}
 \begin{proof} Obvious from Proposition
\ref{lef}.\end{proof}

Assume now that $P$ satisfies the left Ore condition so that $P$ can
be embedded into a group $\bar{P}$ for which $\bar{P}=P^{-1}P$.

Moreover then $\bar{P}$ can be written as $\varinjlim_{p\in P} P$, i.e. as the limit of the
inductive system $(L_p)_{p\in P}$ where $L_p=P$ and the map $L_p\to
L_{qp}$ is given by multiplication by $q$. We set
$\bar{H}=P^{-1}H= {\varinjlim_{p\in P} H}$ with the analogous
inductive limit. \mn

\begin{sproposition} The semigroup $S$ has a natural embedding into the group
$\bar{S}=\bar{H}\rtimes \bar{P}$.\end{sproposition}

We denote by $\rho$ the right regular representation of $S$ on
$\ell^2(S)$ and by $\bar{\rho}$ the right regular representation of
$\bar{S}$ on $\ell^2(\bar{S})$. As above $E$ denotes the orthogonal
projection $\ell^2(\bar{S})\to\ell^2(S)$. We consider $P,\bar{P}$
and $H,\bar{H}$ as subsemigroups of $S,\bar{S}$.

\begin{slemma}\label{S1} Let $(g,z)$ be an element of $\bar{S}$, $g\in
\bar{H}$, $z\in \bar{P}$. Then
$E\bar{\rho}((g,z))E=E\bar{\rho}(z)E\bar{\rho}(g)E$.\end{slemma}
\begin{proof}
Both operators evaluated on an element $\xi_{(a,x)}$ of the standard
orthonormal basis of $\ell^2(S)$ give $\xi_{(ax(h),xz)}$ if $x$ is
such that $ax(h)\in H$, $xg\in P$, and give 0 otherwise.\end{proof}

\begin{slemma}\label{S2} Let $z\in \bar{P}$, $h\in H$ and $g=z(h)\in
\bar{H}$. Assume that \bgl\label{J1}\{x\in P:xz\in P\}\;=\;\{x\in
P:x\,z(h)\in H\}.\egl Then $E\bar{\rho} (g)E=E\bar{\rho} (z)^*E\rho
(h)E\bar{\rho} (z)E$.\end{slemma}

\begin{proof} Let $\xi_{(a,x)}$, $a\in H,\,x\in P$ be an element of the
standard orthonormal basis in $\ell^2(S)$. We have
$$E\bar{\rho}(z)^*E\rho (h) E\bar{\rho} (z)\,\xi_{(a,x)} =
\left\{\begin{matrix}E\bar{\rho} (z)^* E\rho (h)
\,\xi_{(a,\,xz)}=\xi_{(axz(h),\,x)} & \mbox{if}\;
xz\in P,\\[2mm]
0 & \mbox{otherwise.}\end{matrix}\right. $$ On the other hand
$$E\bar{\rho} (g)\,\xi_{(a,x)}=E\bar{\rho} (z(h))\,\xi_{(a,x)}=
\left\{\begin{matrix}\xi_{(axz(h),x)} & \mbox{if}\;
xz(h)\in H, \\[2mm]  0 & \mbox{otherwise.}\end{matrix}\right.$$
\end{proof}

\begin{slemma}\label{S3} Assume that $P$ satisfies the right Toeplitz
condition. Then for each $z\in \bar{P}$ the operator
$E\bar{\rho}(z)E$ is a product of operators of the form $\rho (p)$
or $\rho (p)^*$, $p\in P$.\end{slemma}
\begin{proof} Let $\rho_0$,
$\bar{\rho}_0$ denote the right regular representation of $P$,
$\bar{P}$ on $\ell^2(P)$, $\ell^2(\bar{P})$, respectively, and let
$E_0$ be the orthogonal projection of $\ell^2(\bar{P})$ onto $\ell^2(P)$. By
assumption, there are $p_i,q_i\in P$ such that $E_0\bar{\rho}_0(z)E_0$ is a product of the form
$\rho_0(p_1)\rho_0(q_1)^*\ldots \rho_0(p_n)\rho_0(q_n)^*$.

The Hilbert space $\ell^2(\bar{S})$ has the following filtration by
subspaces
$$\ell^2(S)=\ell^2(H\rtimes P)\subseteq \;L=\ell^2(H\times \bar{P})\; \subseteq
\;\ell^2(\bar{H}\rtimes\bar{P})$$ On the subspace $L\cong
\ell^2(H)\otimes\ell^2(\bar{P})$, the operator $\bar{\rho}(z)$ acts
like $1\otimes\bar{\rho}_0(z)$. Similarly $E$ acts like $1\otimes
E_0$ on $L$. Thus, $E\bar{\rho}(z)E$ acts like
$$1\otimes
\rho_0(p_1)\rho_0(q_1)^*\ldots \rho_0(p_n)\rho_0(q_n)^*$$ and
therefore  {$E\bar{\rho}(z)E=\rho(p_1)\rho(q_1)^*\ldots
\rho(p_n)\rho(q_n)^*$}.\end{proof}

\begin{sproposition}\label{S4} Assume that $P$ satisfies the right Toeplitz
condition and that the following condition is satisfied:
\bgln \label{J-cond}
  & \textrm{for every}\; \bar{h} \in \bar{H} \text{, there exists } z
  \in \bar{P} \text{ and } h \in H \text{ such that } \\
  & \bar{h} = z(h) \text{ and } P \cap Pz^{-1} = \{x \in P:x(\bar{h})
  \in H\}. \nonumber
\egln Then $S\subseteq \bar{S}$ satisfies the right Toeplitz
condition.\end{sproposition}

\begin{proof} This follows by combining Lemmas \ref{S1}, \ref{S2}, \ref{S3}.
In fact, let $(g,w)$ be an element of $\bar{S}$. Then by Lemma
\ref{S1} we have
$E\bar{\rho}((g,w))E=E\bar{\rho}(w)E\bar{\rho}(g)E$. By Lemma
\ref{S3}, $E\bar{\rho}(w)E$ is a product of operators of the form
$\rho (p)$ or $\rho (p)^*$, $p\in P$.

Let finally $z\in \bar{P}$ and $h\in H$ such that $g=z(h)$ and such
that $$\{x\in P:xz\in P\}\;=\;\{x\in P:x\,z(h)\in H\}$$ (this is
equivalent to $P \cap Pz^{-1} = \{x \in P:x(g) \in H\}$). Then, by
Lemma \ref{S2}, $E\bar{\rho} (g)E=E\bar{\rho} (z)^*E\rho
(h)E\bar{\rho} (z)E$. Moreover, by Lemma \ref{S3}, $E\bar{\rho}
(z)E$ and $E\bar{\rho} (z)^*E$ are also products of the desired
form.\end{proof}\mn

\subsection{$K$-theory for the right regular C*-algebra of a
semidirect product} Let us now compute the K-theory of $C^*_\rho
(H\rtimes P)$. We apply our general K-theoretic result switching
from right actions of $H \rtimes P$ to left actions of $(H \rtimes
P)^{op}$. To do so, we need to compute the orbits of the (right)
$\bar{H} \rtimes \bar{P}$-action on the family of constructible left
$H \rtimes P$-ideals in $\bar{H} \rtimes \bar{P}$. If $H \rtimes P
\subseteq \bar{H} \rtimes \bar{P}$ is right Toeplitz and has
independent constructible left ideals, then by Lemma \ref{lem-strongtoeplitz}
applied to $(H \rtimes P)^{op}$, we know that every
constructible left $H \rtimes P$-ideal in $\bar{H} \rtimes \bar{P}$
is in the orbit of a constructible left ideal of $H \rtimes P$. Thus
it suffices to consider constructible left ideals of $H \rtimes P$.
They are of the form $H \times X$ for some constructible left ideal
$X$ of $P$ (see Proposition \ref{lef}). For every such non-empty
$X$, the stabilizer $\{g \in \bar{H} \rtimes \bar{P} \;{:}\; (H \times X)
\cdot g = H \times X\}$ is given by
$$\bar{S}(X):=X^{-1}(H) \rtimes ({}_X
\bar{P})$$
 where $X^{-1}(H) = \{\bar{h} \in \bar{H}: x(\bar{h}) \in H
\,\mbox{for all}\; x \in X\}$ and ${}_X \bar{P} = \{\bar{p} \in
\bar{P} : X \bar{p} = X\}$.

Therefore, combining Corollary \ref{C1}, Proposition \ref{S4} and
Theorem \ref{thm-semiK}, we obtain

\begin{scorollary}\label{C2} Let $\alpha: (\bar{H} \rtimes \bar{P})^{op} \to
\Aut(A)$ be an action of $(\bar{H} \rtimes \bar{P})^{op}$ on a
C*-algebra $A$. Let $\cX$ be a set of constructible left ideals of
$H \rtimes P$ such that $\{H \times X\;{:}\; X \in \cX\}$ is a complete
system of representatives for the $\bar{H} \rtimes \bar{P}$-orbits
of the family of non-empty constructible left $H \rtimes P$-ideals
in $\bar{H} \rtimes \bar{P}$.

Assume that $P$ has independent constructible left ideals, that the
action of $P$ on $H$ satisfies the condition in Proposition \ref{S4}
and that $(\bar{H} \rtimes \bar{P})^{op}$ satisfies the Baum-Connes
conjecture with coefficients. Then we have a canonical isomorphism
 $$K_*(A \rtimes_{\alpha,r} (H \rtimes P)^{op})\quad\cong\quad
  \bigoplus_{X \in \cX} K_*(A \rtimes_{\alpha,r}\bar{S}(X)^{op}).$$
   In particular, we obtain the following
K-theoretic formula for the right regular C*-algebra
$C^*_{\rho}(P)$:
$$K_*(C^*_{\rho}(H \rtimes P))\quad\cong\quad
\bigoplus_{X \in \cX} K_*(C^*_\lambda(\bar{S}(X))).
$$

 \end{scorollary}\mn

\subsection{Right and left C*-algebras for semidirect products by $\Nz$}
To deduce the right Toeplitz condition for $H \rtimes P \subseteq
\bar{H} \rtimes \bar{P}$, we used in Proposition \ref{S4} condition
(\ref{J-cond}) which says: for every $\bar{h} \in \bar{H}$ there
exists $z\in \bar{P}$ and $h \in H$ such that $\bar{h} = z(h)$ and
$P \cap Pz^{-1} = \{x \in P:x(\bar{h})
  \in H\}$. Note that the set $\{x \in P:x(\bar{h}) \in H\}$ is always a left
ideal of $P$. Moreover, we have:
\begin{slemma} \label{principal}
\eqref{J-cond} holds if for every $\bar{h}$ in $\bar{H}$, the left
ideal $\{x \in P:x(\bar{h}) \in H\}$ is principal, i.e. of the form
$Pp$ for some $p \in P$. \end{slemma}
 \begin{proof} If $\{x \in P:x(\bar{h})
\in H\} = Pp$, then $p$ itself satisfies $p(\bar{h}) \in H$. Thus
there exists $h \in H$ such that $\bar{h} = p^{-1}(h)$, and setting
$z = p^{-1}$, we see that \eqref{J-cond} is satisfied. \end{proof} In
general, it is not clear which left ideals of $P$ arise as sets of
the form $\{x \in P:x(\bar{h}) \in H\}$ for some $\bar{h} \in
\bar{H}$. So in general, we can only deduce the following
\begin{scorollary}
Assume that all non-empty left ideals of $P$ are principal. Then for
every action of $P$ on some group $H$, condition~\eqref{J-cond}
holds true. In particular, \eqref{J-cond} holds for every
$\Nz$-action. \end{scorollary}

\begin{slemma}\label{S5} Assume that the (additive) semigroup $\Nz$ acts
by injective endomorphisms $\alpha_n$ on the group $H$. The set of
constructible right ideals in $H\rtimes \Nz$  coincides with the set
of principal ideals. The principal right ideals are exactly the
subsets of the form $h\alpha_n (H)\times (n)$ with $h\in
H/\alpha_n(H)$, where $(n)=\{n+k:k\in \Nz\}$ denotes the principal
ideal generated by $n$ in $\Nz$.\end{slemma}
 \begin{proof} The principal
right ideals in $H\rtimes \Nz$ are of the form $(h,n)(H\rtimes
\Nz)=h\alpha_n(H)\times (n)$. We show that the set of principal
ideals is closed under the operation $I\mapsto (g,k)^{-1}I$. One easily checks that

$$ \begin{array}{l} (g,k)^{-1}(h\alpha_n(H)\times (n))
=\left\{\begin{matrix}H\times\Nz & \mbox{if}\, g^{-1}h\in
\alpha_n(H),\\[2mm]\emptyset & \textrm{otherwise,}\end{matrix}\right.\\[7mm]
(g,k)^{-1}(h\alpha_n(H)\times (n))
=\left\{\begin{matrix}a\alpha_{n-k}(H)\times (n-k) &\mbox{if}\;
g^{-1}h=\alpha_k(a),\,a\in H,\\[2mm]
\emptyset &
\textrm{if}\; g^{-1}h\notin \alpha_k
(H),\end{matrix}\right.\end{array}$$
 depending if $k\geq n$ in the first case or $k\leq n$ in the second one.\end{proof}

\begin{stheorem}\label{T1} Let $\Nz$ act by injective endomorphisms on the
group $H$ {and assume that the enveloping group $\bar{H}\rtimes\bar{\NN}=\bar{H}\rtimes \ZZ$ satisfies the 
Baum-Connes conjecture with coefficients}. Let $C^*_\rho (H\rtimes \Nz)$ and $C^*_\lambda (H\rtimes
\Nz)$ denote the right and left regular C*-algebra of $H\rtimes
\Nz$, respectively. Then
\begin{itemize}
\item [(a)]  $K_*(C^*_\rho (H\rtimes \Nz))\cong
K_*(C^*_\lambda(H))$.
\item [(b)] $K_*(C^*_\lambda (H\rtimes \Nz))\cong K_*(C^*_\lambda(H))$.
\end{itemize}
In particular $C^*_\rho (H\rtimes \Nz)$ and $C^*_\lambda (H\rtimes
\Nz)$ have the same $K$-theory.\end{stheorem}

\begin{proof} (a) is an immediate consequence of the description of the
constructible left ideals of $H\rtimes \Nz$ in Proposition \ref{lef}
together with the formula for $K_*$ in Corollary \ref{C2}.

(b) The set $\cJ$ of constructible right ideals in $H\rtimes \Nz$ is
described in Lemma \ref{S5}. It is obviously independent. Every
ideal in $\cJ$ has full orbit under the action of $\bar{H}\rtimes
\Zz$ and the stabilizer group of the ideal $H\rtimes \Nz$ is $H$.
Thus the assertion follows from Theorem \ref{thm-semiK}.\end{proof}

It is by no means obvious that $C^*_\rho (H\rtimes \Nz)$ and
$C^*_\lambda (H\rtimes \Nz)$ should have the same $K$-theory. In
fact, as C*-algebras they look very different from each other. For
instance in the situation of the following remark ($H$ abelian and
$H/\vp (H)$ finite) $C^*_\rho (H\rtimes \Nz)$ admits non-trivial
one-dimensional representations (see \cite{Li-am}) while every
non-zero quotient of $C^*_\lambda (H\rtimes \Nz)$ contains a
non-trivial isometry and therefore $C^*_\lambda (H\rtimes \Nz)$
admits no non-trivial finite-dimensional representations.

\begin{sremark} Consider the special case where $H$ is abelian and $\Nz$
acts via the injective endomorphism $\vp$ on $H$. Assume also that
$H/\vp (H)$ is finite and that $\bigcap_{n\geq 0}\vp^n(H)=\{0\}$.
Then the left regular C*-algebra $C^*_\lambda(H\rtimes \Nz)$ has the
algebra $\cA [\vp]$ studied in \cite{CV} as natural quotient. In
\cite{CV} it was shown that the $K$-theory of $\cA [\vp]$ is
determined by a six term exact sequence of the form
\bgl\label{exPV}\xymatrix{K_*C^*(H)\ar[r]^{1-b(\vp)}&
K_*C^*(H)\ar[r]&K_*\Af\ar@/^5mm/[ll]}\egl\mn On the other hand, we
know by Theorem \ref{T1} that $K_*(C^*_\lambda(H\rtimes \Nz))\cong
K_*(C^*(H))$. It can be shown that the long exact sequence
associated with the extension
$$ 0\to \Ker \pi \to C^*(H\rtimes \Nz)
\stackrel{\pi}{\longrightarrow} \cA [\vp]\to 0$$ is exactly the
exact sequence in (\ref{exPV}).\end{sremark}\mn

\subsection{Right and left C*-algebras for $ax+b$-semigroups of Dedekind rings}
Recall that a Dedekind ring is a noetherian integrally closed
integral domain in which every non-zero prime ideal is maximal. The
prime example of a Dedekind ring is the ring of algebraic integers
in a number field. If $R$ is a Dedekind ring, its $ax+b$-semigroup
is, by definition, the semidirect product $R\rtimes R^\times$ where
$R^\times =R\backslash \{0\}$ denotes the multiplicative semigroup
of the ring and $R$ (by abuse of notation) its additive group.

\begin{sproposition} Let $R$ be a Dedekind domain with field of fractions $Q$.
Then the inclusion of $ax+b$-semigroups $R\rtimes R^\times\subseteq
Q\rtimes Q^\times$ satisfies the right Toeplitz condition.\end{sproposition}
\begin{proof} We apply Proposition \ref{S4} with $H=R$, $P=R^\times$.
Since the inclusion $R^\times\subseteq Q^\times$ satisfies the left
Toeplitz condition, by commutativity it also satisfies the right
condition. Moreover, given $0\neq \bar{h}\in Q$, choose $h=1$ and
$z=\bar{h}$. These elements obviously have the properties required
in Proposition \ref{S4}.\end{proof}

\begin{stheorem}  Let $R^*$ be the group of units in $R$ and choose for every
ideal class $\gamma \in Cl_{Q(R)}$ an ideal $I_{\gamma}$ of $R$
which represents $\gamma$. The $K$-theory of the right regular
C*-algebra $C^*_\rho (R\rtimes R^\times )$ is given by the formula
$$
  K_*(C^*_{\rho} (R \rtimes R^\times))\cong \bigoplus_{\gamma
  \in Cl_{Q(R)}} K_*(C^*_\lambda(I_{\gamma}^{-1} \rtimes R^*)).
$$ Here we use the notation, familiar from number theory, $$I_\gamma^{-1}
=\{x\in Q:xy\in R,\; \forall y\in I_\gamma\}$$ for the fractional
ideal $I_\gamma^{-1}$ in the quotient field $Q$ of $R$ (it satisfies
$I_\gamma^{-1}I_\gamma =R$). \end{stheorem}
\begin{proof} By Proposition
\ref{lef}, the constructible left ideals of $R\rtimes R^\times$ are
in bijection with the constructible ideals of $R^\times$. These
ideals are exactly of the form $I^\times$ where $I$ is a ring ideal
in $R$, see \cite{CEL}. The orbits under the action of the
enveloping group of $R^\times$ are labeled by the elements $\gamma$
of the class group $Cl_{Q(R)}$. According to the discussion before
Corollary \ref{C2} the stabilizer group for such an element $\gamma$
is $I_\gamma^{-1}\rtimes R^*$. The assertion now follows from
Corollary \ref{C2}.\end{proof}

In particular, comparing with the result obtained in \cite{CEL}, we
see that the left and right regular C*-algebras of $R \rtimes
R^\times$ are $KK$-equivalent (they both are $KK$-equivalent to
$\bigoplus_{\gamma} C_\lambda^*(I_{\gamma}^{-1}\rtimes R^*)$).
 Again, this is by no means obvious
since $C^*_{\rho} (R \rtimes R^\times)$ and $C^*_{\lambda} (R
\rtimes R^\times)$ are quite different (again the first algebra
admits non-trivial one-dimensional representations while by
\cite{CDL} the second one admits a largest ideal (which contains any
other non-trivial ideal) with a simple quotient (the ring C*-algebra
of \cite{CuLi})).

\subsection{Wreath products}
We here discuss another important class of specific semidirect
products, so-called wreath products. For this we take a left Ore
semigroup $P$, a group $\Gamma$ with unit $e$ and form
$$
  \Gamma_\infty^{ P} = \bigoplus_{x \in P} \Gamma = \{f: P \to \Gamma \ : \ f(x) = e \text{ for almost all } x \in P\}.
$$
$P$ acts on $\Gamma_\infty^P$ by shifting from the left, i.e.
$p(f)(x) = e$ if $x \notin pP$ and $p(f)(x) = f(p^{-1}x)$ if $x \in
pP$. Let $\Gamma \wr P = \Gamma_\infty^P \rtimes P$ be the
semidirect product attached to this action of $P$ on
$\Gamma_\infty^P$. The semigroup $\Gamma \wr P$ is in a canonical
way a subsemigroup of $\Gamma \wr \bar{P}$ (with $\bar{P} = P^{-1}
P$).

We first consider the left regular representation. Let $\cJ_P$,
$\cJ_{\Gamma \wr P}$ be the families of constructible right ideals
in $P$, $\Gamma \wr P$, respectively. Then we have
\begin{slemma}
$\cJ_{\Gamma \wr P} = \{(f \cdot (\Gamma_\infty^X)) \times X:f \in
\Gamma_\infty^P, X \in \cJ_P\}$. \end{slemma}
\begin{proof} It is clear that
the right hand side contains $\emptyset$, $\Gamma \wr P$ and that it
is closed under left multiplication. Moreover, given $f \in
\Gamma_\infty^P$, $X \in \cJ_P$ and $(h,p) \in \Gamma \wr P$, either
$(h,p)^{-1}((f \cdot (\Gamma_\infty^X)) \times X)$ is empty or there
is $\ti{f} \in \Gamma_\infty^P$ with $hp(\ti{f}) \in f \cdot
(\Gamma_\infty^X)$. In the latter case, it is immediate that
$$
  (h,p)^{-1}((f \cdot (\Gamma_\infty^X)) \times X) = (\tilde{f} \cdot
  (\Gamma_\infty^{p^{-1}X})) \times p^{-1}X.
$$
\end{proof}

 \begin{scorollary} If $\cJ_P$ is independent, then $\cJ_{\Gamma \wr P}$
is independent. \end{scorollary}
\begin{proof} Assume that we have
$$
  (f \cdot (\Gamma_\infty^X)) \times X = \bigcup_{i=1}^n (f_i \cdot
  (\Gamma_\infty^{X_i})) \times X_i
$$
for some $f$, $f_1$, ..., $f_n$ in $\Gamma_\infty^P$ and $X$, $X_1$,
..., $X_n$ in $\cJ_P$. Projecting down onto the $P$-coordinate, we
see that $X = \bigcup_{i=1}^n X_i$. Hence by independence of
$\cJ_P$, we must have $X = X_i$ for some $i$. Therefore,
$\Gamma_\infty^X = \Gamma_\infty^{X_i}$. But because $f \cdot
(\Gamma_\infty^X)$ and $f_i \cdot (\Gamma_\infty^X)$ are either
equal or disjoint, we deduce that $(f \cdot (\Gamma_\infty^X))
\times X = (f_i \cdot (\Gamma_\infty^{X_i})) \times X_i$. \end{proof}\mn

We now turn to the right regular representation. In the particular
situation of the action of $P$ on $\Gamma_\infty^P$, we can say a
bit more about condition~\eqref{J-cond} in Proposition \ref{S4}.
Namely, take $f \in \Gamma_\infty^{\bar{P}}$ with support $\{x_1,
\dotsc, x_n\}$, i.e. $f(x) = e$ whenever $x \in \bar{P} \setminus
\{x_1, \dotsc, x_n\}$ and $f(x_i) \neq e$ for all $1 \leq i \leq n$.
Then
$$
  \{p \in P:p(f) \in \Gamma_\infty^P\} = P \cap \bigcap_{i=1}^n P x_i^{-1}.
$$
Therefore, the ideals which arise as sets of the form $\{p \in
P:p(f) \in \Gamma_\infty^P\}$ are precisely the constructible left
ideals of $P$ if $P \subseteq\bar{P}$ is assumed to be right
Toeplitz (see \cite[Lemma~4.2]{Li-nuc}). So by Lemma~\ref{principal}
and Proposition \ref{S4}, we deduce
 \begin{scorollary} If $P \subseteq \bar{P}$
is right Toeplitz and all the constructible left ideals of $P$ are
principal, then $\Gamma \wr P \,\subseteq\, \Gamma \wr \bar{P}$ is
right Toeplitz. \end{scorollary}

As a particular example, take $\Gamma = \Zz / 2 \Zz$ and $P = \Nz$. Then the enveloping group of $\Gamma \wr P$ is the classical lamplighter group $(\Zz / 2 \Zz) \wr \Zz$. To compute K-theory, we can simply apply Theorem~\ref{T1}, and we obtain
$$
  K_*(C^*_{\lambda}((\Zz / 2 \Zz) \wr \Nz)) \cong K_*(\bigotimes_{i=1}^{\infty} C^*(\Zz / 2 \Zz)) \cong K_*(C^*_{\rho}((\Zz / 2 \Zz) \wr \Nz)).
$$

\section{Appendix: A remark on equivariant $K$-theory for finite dimensional commutative C*-algebras}

Suppose that $C$ and $B$ are finite dimensional commutative C*-algebras, i.e., there exist positive integers $n$ and $m$ such that $C\cong \CC^n$ and $B\cong \CC^m$ and we may choose {bases}
of pairwise orthogonal projections $\{c_1,\ldots, c_n\}$ and $\{b_1,\ldots, b_m\}$ of {$C$} and $B$.

Recall that by the UCT-theorem we have
isomorphisms
$$KK(C,B)\cong \Hom(K_0(C), K_0(B))\cong M(m\times n, \ZZ)$$
where the first one is given by sending a class $x\in KK(C,B)$ to the associated
homomorphism $[\cdot]{\otimes}_Cx:K_0(C)\to K_0(B)$ and the second one is given by
describing this map with respect to the canonical generators
$\{[c_1],\ldots, [c_n]\}$ and $\{[b_1],\ldots, [b_m]\}$ of $K_0(C)$ and $K_0(B)$, respectively, i.e.,
the matrix $\Gamma=(\gamma_{ij})$ corresponding to $x$ is determined by
$$[c_j]{\otimes}_Cx=\sum_{i=1}^m\gamma_{ij}[b_i]$$
for all $1\leq j\leq n$.

Let us describe how we may construct for a given matrix $\Gamma\in M(m\times n,\ZZ)$ the corresponding class $x_\Gamma\in KK(C,B)$. For this we first decompose $\Gamma$ as the difference
$\Gamma=\Gamma^+-\Gamma^-$  where $\Gamma^+$ is the matrix built out of $\Gamma$ by replacing  all negative entries by $0$ and $\Gamma^-:=\Gamma^+-\Gamma$.
We then construct a graded Kasparov module $\mathcal E=\mathcal E^+\oplus \mathcal E^-$
with
$$\mathcal E^+=\bigoplus_{j=1}^n\left(\bigoplus_{i=1}^m\big(\CC^{\gamma_{ij}^+}\otimes B_i\big)\right)
\quad\text{and}\quad
\mathcal E^-=\bigoplus_{j=1}^n\left(\bigoplus_{i=1}^m\big(\CC^{\gamma_{ij}^-}\otimes B_i\big)\right)$$
equipped with the canonical $B$-valued inner products, where $B_i=\CC b_i\subseteq B$ denotes the ideal generated by $b_i$. For each $1\leq j\leq n$ let $p_j^+\in \K(\E^+)$ denote the orthogonal projection on the $j$th summand $\bigoplus_{i=1}^m\big(\CC^{\gamma_{ij}^+}\otimes B_i\big)$ of $\E^+$, and,
similarly, we let $p_j^-\in \K(\E^-)$ denote the orthogonal projection onto the  $j$th summand
$\bigoplus_{i=1}^m\big(\CC^{\gamma_{ij}^-}\otimes B_i\big)$ of $\E^-$.
We then define a homomorphism
$$\varphi^+:C\to \K(\mathcal E^+);\;\; \varphi^+\big(\sum_{j=1}^n \lambda_kc_j\big)=
\sum_{i=1}^n \lambda_jp_j^+$$
and, in a similar way,  we define the homomorphism $\varphi^-:C\to \K(\mathcal E^-)$. Then one easily checks that
$$x_\Gamma=[(\mathcal E^+\oplus \mathcal E^-, \varphi^+\oplus \varphi^-,0)]\in KK(C,B)$$
 is the  class corresponding to $\Gamma\in M(m\times n, \ZZ)$.

Suppose now that $G$ is a locally compact group which acts on $C$ and $B$ via permutations of the bases
$\{c_1,\ldots,c_n\}$ and $\{b_1,\ldots, b_m\}$, respectively. Let $\mu_C:G\to S_n$ and $\mu_B:G\to S_m$ denote the
corresponding homomorphisms into the permutation groups $S_n$ and $S_m$, respectively. We
shall often simply write $g \cdot j$ (resp. $ g\cdot i$) for $\mu_C(g)(j)$ (resp. $\mu_B(g)(i)$). We note that these actions will always factor through actions of some finite quotient $G/N$ of $G$, so that in the  following discussion one could assume as well that $G$ is finite.
\medskip

For $x\in KK^G(C,B)$, the corresponding element in $\Hom(K_0(C), K_0(B))$ is equivariant with respect to the  actions of $G$ on $K_0(C)$ and $K_0(B)$ induced by the given actions on $C$ and $B$, respectively. This implies that the corresponding matrix $\Gamma\in M(m\times n,\ZZ)$ satisfies the relation $\Gamma\circ \mu_C(g)=\mu_B(g)\circ \Gamma$ for all $g\in G$. This easily
translates to the condition $\gamma_{g\cdot i, g\cdot j}=\gamma_{ij}$ for each entry $\gamma_{ij}$ of $\Gamma$. It therefore follows that the same
relations  hold for $\Gamma^+$ and $\Gamma^-$ and we may define an action
$\mu_\E:G\to \Aut(\mathcal E^{+/-})$ by
$$\mu_\E(g)\left(\bigoplus_{j=1}^n\big(\bigoplus_{i=1}^m v_{ij}\otimes b_i\big)\right)
=\bigoplus_{j=1}^n\big(\bigoplus_{i=1}^m v_{g^{-1}\cdot i, g^{-1}\cdot j}\otimes b_i)\big).$$

Let us check that
$\varphi=(\varphi^+,\varphi^-): C\to \K(\mathcal E^+\oplus\mathcal E^-)$ is $G$-equivariant.
For this let $c_l$ be a fixed basis element of $C$. We want to compare $\varphi^+(\mu_C(g)( c_l))$
with $\mu_\E(g)\varphi^+(c_l)\mu_\E(g^{-1})$ and we do this by computing what both operators do to the
{$(i,j)$-th} summand $\CC^{\gamma_{ij}^+}\otimes B_i$ of $\E^+$. First of all, the projection
$\varphi^+(\mu_C(g)( c_l))=p_{{g \cdot l}}^+$ fixes the element $v_{ij}\otimes b_i$ if  $j= {g \cdot l}$ and sends it to $0$ if $j\neq {g \cdot l}$.
 In order to compute $\mu_\E(g)\varphi^+(c_l)\mu_\E(g^{-1})(v_{ij}\otimes b_i)$ we first
observe that $\mu_\E(g^{-1})$ moves $v_{ij}\otimes b_i$ to the element
$v_{ij}\otimes b_{{g^{-1} \cdot i}}$ at the {$(g^{-1}\cdot i, g^{-1}\cdot j)$-th} place. Then $\varphi^+(c_l)=p_l^+$ will
fix this element  if $l=g^{-1}\cdot j$ (i.e. $j= {g \cdot l}$) and will send it to $0$ else. Finally $\mu_\E(g)$ will move
$v_{ij}\otimes b_{{g^{-1} \cdot i}}$
to the element $v_{ij}\otimes b_i$ at the {$(i,j)$-th} place
if $j={g \cdot l}$. This shows the desired result. The same computation yields equivariance of $\varphi^-$.

\begin{remark}\label{rem-decomp}
We could have constructed the same element by any other decomposition $\Gamma=\tilde{\Gamma}^+-\tilde{\Gamma}^-$ as long as both matrices
$\tilde{\Gamma}^+,\tilde{\Gamma}^-$ only have positive integer
entries and satisfy the relations $\tilde\gamma_{{g \cdot i}, {g \cdot j}}^{+/-}=\tilde\gamma_{ij}^{+/-}$. In fact,  if we do the construction with the help of such an alternative decomposition to obtain
a class $\tilde{x}_\Gamma=[(\tilde\E^+\oplus \tilde\E^-, \tilde\varphi^+\oplus \tilde\varphi^-,0)]$,
then the difference
$\tilde{x}_\Gamma-x_{\Gamma}$ is represented by the Kasparov triple
$[(\F^+\oplus \F^-, \psi^+\oplus \psi^-, 0)]$
with
$$\F^+=\tilde\E^+\oplus\E^-, \;\F^-=\tilde\E^-\oplus \E^+\quad\text{and}
\quad
\psi^+=\tilde\varphi^+\oplus \varphi^-, \;\psi^-=\tilde\varphi^-\oplus\varphi^+.$$
Using the equation $\Gamma^++\tilde{\Gamma}^-=\Gamma^-+\tilde{\Gamma}^+$, one checks that
$$\F^+=\F^-=\bigoplus_{j=1}^n\left(\bigoplus_{i=1}^m\big(\CC^{\tilde{\gamma}_{ij}^++\gamma_{ij}^-}\otimes B_i\big)\right)\quad\text{and}\quad \psi^+=\psi^-,$$
which implies that the triple $(\F^+\oplus \F^-, \psi^+\oplus \psi^-, 0)$
is operator homotopic to the degenerate triple $\left(\F^+\oplus \F^-, \psi^+\oplus \psi^-, \left(\begin{smallmatrix}
0&1\\1&0\end{smallmatrix}\right)\right)$ via $t\mapsto t\left(\begin{smallmatrix}
0&1\\1&0\end{smallmatrix}\right)$.
\end{remark}

We shall need

\begin{lemma}\label{lem-compose}
Suppose that $C=\CC^n, B=\CC^m, A=\CC^k$ are equipped with  actions of the locally compact group
$G$  given by homomorphisms $\mu_C:G\to S_n, \mu_B:G\to S_m$, and  $\mu_A:G\to S_k$.
Let $\Gamma\in M(m\times n,\ZZ)$ and $\Lambda\in M(n\times k,\ZZ)$ such that
$$\Gamma\circ \mu_C(g)=\mu_B(g)\circ \Gamma\quad\text{and}\quad {\Lambda\circ \mu_B(g)=\mu_A(g)\circ \Lambda}$$
for all $g\in G$. Then $x_\Gamma^G\otimes_Bx_{\Lambda}^G=x_{\Lambda\cdot\Gamma}^G$ in $KK^G(C,A)$. In particular, if $m=n$ and $\Gamma\in \GL(n,\ZZ)$, then $x_\Gamma^G\in KK^G(C,B)$
is invertible with inverse given by the class $x_{\Gamma^{-1}}^G\in KK^G(B,C)$.
\end{lemma}
\begin{proof} Let $(\E^+\oplus\E^-, \varphi^+\oplus \varphi^-,0)$ and $(\F^+\oplus\F^-,\psi^+\oplus \psi^-,0)$ denote the corresponding Kasparov triples as constructed above from $\Gamma$ and $\Lambda$. Then the product $x_\Gamma\otimes_Bx_\Lambda$ is represented by the triple
$(\G^+\oplus\G^-,\mu^+\oplus \mu^-, 0)$ with
$$\G^+=(\E^+\otimes_B\F^+)\oplus (\E^-\otimes_B\F^-)\quad \text{and}\quad
\G^-=(\E^+\otimes_B\F^-)\oplus(\E^-\otimes \F^+)$$ and with
$\mu^+=(\varphi^+\otimes 1_{\F^+})\oplus (\varphi^-\otimes
1_{\F^{-}})$ and $\mu^-=(\varphi^+\otimes 1_{\F^-})\oplus
(\varphi^-\otimes 1_{\F^+})$. Of course, these modules decompose
into summands of the form $\big(\CC^{\gamma_{ij}^{+/-}}\otimes
B_i\big)\otimes_B\big(\CC^{\lambda_{lr}^{+/-}}\otimes A_l\big)$,
where $A_l=\CC a_l$ for $a_l$ an element  in a given basis
$\{a_1,\ldots, a_k\}$ of pairwise orthogonal projections of $A$, and
we now compute these summands: Since $b_i^2=b_i$, the balanced
tensor product $\big(\CC^{\gamma_{ij}^{+/-}}\otimes
B_i\big)\otimes_B\big(\CC^{\lambda_{lr}^{+/-}}\otimes A_l\big)$ is
generated by elementary tensors $(v_{ij}\otimes b_i)\otimes
(w_{lr}\otimes a_r)$ modulo the relation
$$(v_{ij}\otimes b_i)\otimes (w_{lr}\otimes a_r)=(v_{ij}\otimes b_i)\otimes \psi^{+/-}(b_i)(w_{lr}\otimes a_l)$$
which forces the element to be zero if $r\neq i$, and which is always satisfied if $r=i$. Thus we see
that
$$\big(\CC^{\gamma_{ij}^{+/-}}\otimes B_i\big)\otimes_B\big(\CC^{\lambda_{lr}^{+/-}}\otimes A_l\big)
\cong \left\{\begin{matrix} \CC^{\gamma_{ij}^{+/-}\cdot \lambda_{li}^{+/-}}\otimes A_l&
\text{if $i=r$}\\ 0&\text{if $i\neq r$}\end{matrix}\right..$$
Moreover, the projections $\mu^{+/-}(c_t)$ will fix these spaces if and only if $t=j$ and will send them to $0$ otherwise. Summing up over $i$ and using Remark \ref{rem-decomp} then shows that  $[(\G^+\oplus\G^-,\mu^+\oplus \mu^-, 0)]$ equals $x_{\Lambda\cdot \Gamma}^G$.

\end{proof}


\begin{thebibliography}{10}

 \bibitem{Black} B. Blackadar, \emph{$K$-theory for operator algebras}. Second edition. Mathematical Sciences Research Institute Publications  {\bf 5},
 Cambridge University Press, Cambridge, 1998.

 \bibitem{BCH} P. Baum, A. Connes and N. Higson, \emph{Classifying space for proper actions and $K$-theory of group $C^*$-algebras},
 Contemporary Mathematics {\bf 167} (1994), 241-291.

% \bibitem{BMP} P. Baum, S. Millington and R. Plymen. \emph{Local-global principle for the Baum-Connes conjecture with coefficients.} K-Theory {\bf 28} (2003), no. 1, 1--18.

 \bibitem{CE} J. Chabert and  S. Echterhoff,
 \emph{Permanence properties of the Baum-Connes conjecture.} Doc. Math. {\bf 6} (2001), 127--183.

 \bibitem{CEN} J. Chabert, S. Echterhoff and R. Nest, \emph{The Connes-Kasparov conjecture for almost connected groups and for linear $p$-adic groups},
 Publ. Math. Inst. Hautes \'Etudes Sci. No. {\bf 97} (2003), 239--278.

 \bibitem{CEO} J. Chabert, S. Echterhoff and H. Oyono-Oyono,
 \emph{Going-down functors, the K\"unneth formula, and the Baum-Connes conjecture}, Geom. Funct. Anal. {\bf 14} (2004), 491--528.

 \bibitem{Cr-La} J. Crisp and M. Laca,
 \emph{On the Toeplitz algebras of right-angled and finite-type Artin groups}, J. Austral. Math. Soc. {\bf 72} (2002), 223--245.

 \bibitem{Cun} J. Cuntz,
 \emph{$K$-theoretic amenability for discrete groups}, J. Reine Angew. Math. {\bf 344} (1983), 180--195.

 \bibitem{CDL} J. Cuntz, C. Deninger and M. Laca,
 \emph{$C^*$-algebras of Toeplitz type associated with algebraic number fields}, arXiv:1105.5352v2, to appear in Math. Ann.

 \bibitem{CEL} J. Cuntz, S. Echterhoff and X. Li, \emph{$K$-theory for semigroup C*-algebras}, arXiv:1201.4680v1.

 \bibitem{CuLi} J. Cuntz and X. Li,
 \emph{The regular {$C\sp{\ast} $}-algebra of an integral domain}, Clay Mathematics Proceedings {\bf 11}, 149--170, 2010.

 \bibitem{CV} J. Cuntz and A. Vershik,
 \emph{C*-algebras associated with endomorphisms and polymorphisms of compact abelian groups}, arXiv:1202.5960, 2012.

 \bibitem{ELPW} S. Echterhoff, W. L\"uck, N. C. Phillips and S. Walters,
 \emph{The structure of crossed products of irrational rotation algebras by finite subgroups of $SL_2(\ZZ)$}, J. reine angew. Math {\bf 639} (2010), 173--221.

 \bibitem{ENO} S. Echterhoff, R. Nest and H. Oyono-Oyono,
 \emph{Fibrations with noncommutative fibres,} Journal of Noncommutative Geometry {\bf 3} (2009), 377--417.

 \bibitem{Green} P. Green,
 \emph{The local structure of twisted covariance algebras}, Acta. Math. {\bf 140} (1978), 191--250.

{\bibitem{GHW} E. Guentner, N. Higson and S. Weinberger, 
\emph{The Novikov conjecture for linear groups,}
 Publ. Math. Inst. Hautes \'Etudes Sci. {\bf 101} (2005), 243--268.}

 \bibitem{HK} N. Higson and G. Kasparov,
 \emph{ $E$-theory and $KK$-theory for groups which act properly and isometrically on Hilbert space}, Invent. Math. {\bf 144} (2001), 23--74.

 \bibitem{HLS} N. Higson, V. Lafforgue  and G. Skandalis,
 \emph{Counterexamples to the Baum-Connes conjecture}, Geom. Funct. Anal. {\bf 12} (2002), no. 2, 330--354.

 \bibitem{Ivanov} N. Ivanov,
 \emph{The K-theory of Toeplitz C*-algebras of right-angled Artin groups}, Trans. Amer. Math. Soc. {\bf 362} (2010), no. 11, 6003--6027.

 \bibitem{Kas2} G. Kasparov,
 \emph{Equivariant $KK$-theory and the Novikov conjecture}, Invent. Math. {\bf 91} (1988), 147--201.
 
{\bibitem{Put} I. Putnam,
 \emph{The C*-algebras associated with minimal homeomorphisms of the Cantor set}, Pacific J. Math. {\bf 136} no. 2 (1989), 329--353.}

 \bibitem{QS} J. Quigg and J. Spielberg,
 \emph{Regularity and hyporegularity in C*-dynamical systems. } Houston J. Math. {\bf 18} (1992), 139--152.

 \bibitem{Laca} M. Laca,
 \emph{From endomorphisms to automorphisms and back: dilations and full corners}, J. London Math. Soc. {\bf 61} (2000), 893--904.

 \bibitem{Li-am} X. Li,
 \emph{Semigroup C*-algebras and amenability of semigroups}, arXiv:1105.5539v2, to appear in J. Functional Analysis.

 \bibitem{Li-nuc} X. Li,
 \emph{Nuclearity of semigroup C*-algebras and the connection to amenability}, arXiv:1203.0021v2.

 \bibitem{MN} R. Meyer and R. Nest,
 \emph{The Baum-Connes conjecture via localisation of categories}, Topology {\bf 45} (2006), no. 2, 209--259.

 \bibitem{Ni1} A. Nica,
 \emph{C*-algebras generated by isometries and Wiener-Hopf operators}, J. Operator Theory {\bf 27}  (1992), 17--52.

 \bibitem{PV} M. Pimsner and D. Voiculescu,
 \emph{Exact sequences for {$K$}-groups and {E}xt-groups of certain cross-product {$C^{\ast} $}-algebras}, J. Operator Theory {\bf 4} (1980) 93--118.
 
{\bibitem{Ried} N. Riedel,
 \emph{Classification of the C*-algebras associated with minimal rotations.} Pacific J. Math. {\bf 101} no. 1 (1982), 153--161.}

 \bibitem{Tak} M. Takesaki,
 \emph{Covariant representations of C*-algebras and their locally compact automorphism groups}, Acta Math. {\bf 119} (1967), 273--303.

 \bibitem{Tu} J.-L. Tu,
 \emph{La conjecture {de} Baum-Connes pour les feuilletages moyennable}, $K$-theory {\bf 17} (1999), 215-264.


\end{thebibliography}
\end{document}